\documentclass[11pt,twoside,reqno]{amsart}
\usepackage[foot]{amsaddr}
\usepackage{amssymb,amsfonts,amsmath,amsthm}
\usepackage{calrsfs}
\usepackage[mathscr]{eucal}
\usepackage{bbm}
\usepackage{stmaryrd}
\usepackage{fullpage}
\usepackage{blkarray}
\usepackage{array}

\theoremstyle{definition}
\newtheorem{thm}{Theorem}
\newtheorem*{thm*}{Theorem}
 
\newtheorem{lem}[thm]{Lemma}
\newtheorem{defi}[thm]{Definition} 
\newtheorem{cor}[thm]{Corollary}
\newtheorem{rem}[thm]{Remark}

\numberwithin{equation}{section}
\numberwithin{thm}{section}

\usepackage{enumerate}
\usepackage[shortlabels]{enumitem}

\usepackage{hyperref}
\usepackage[dvipsnames]{xcolor}
\newcommand\myshade{85}
\colorlet{mylinkcolor}{red}
\colorlet{mycitecolor}{blue}
\colorlet{myurlcolor}{Aquamarine}
\hypersetup{
	linkcolor  = mylinkcolor,
	citecolor  = mycitecolor,
	urlcolor   = myurlcolor!\myshade!black,
	colorlinks = true,
}

\usepackage{tikz}
\newcommand{\tzl}[1]{\scriptsize{$#1$}}
\usetikzlibrary{arrows}

\newcommand{\la}[1]{\mathfrak{#1}}

\newcommand{\ZZ}{\mathbb{Z}}

\newcommand{\QQ}{\mathbb{Q}}
\newcommand{\CC}{\mathbb{C}}

\newcommand{\cP}{\mathcal{P}}
\newcommand{\cR}{\mathcal{R}}
\newcommand{\cU}{\la{U}}

\newcommand{\sO}{\mathscr{O}}
\newcommand{\sW}{\mathscr{W}}

\newcommand{\hI}{\widehat{I}}

\newcommand{\res}{\mathrm{Res}}

\newcommand{\lz}{\mathopen{\llbracket}}
\newcommand{\rz}{\mathclose{\rrbracket}}

\newcommand{\ourla}{\la{sl}_{2l+1}}
\newcommand{\ourlev}{-l-\frac{1}{2}}

\newcommand{\vac}{\mathbf{1}}

\newcommand{\re}{\mathrm{re}}
\newcommand{\im}{\mathbbm{i}}
\newcommand{\ad}{\mathrm{ad}}

\newcommand{\rs}{\mathrm{short}}
\newcommand{\ri}{\mathrm{intermediate}}
\newcommand{\rl}{\mathrm{long}}

\newcommand{\modn}{\cU(\la{g})\la{n}_+}
\newcommand{\modnz}{\cU(\la{g}^0)\la{n}_+^0}

\allowdisplaybreaks

\begin{document}

\date{}
\title{$A_{2l}^{(2)}$ at level $-l-\frac{1}{2}$}
\author{Shashank Kanade}
\address{University of Denver, Denver, USA} 
\email{\texttt{shashank.kanade@du.edu}}
\thanks{I am extremely grateful to T.\ Creutzig, A.\ Linshaw and D.\ Ridout for many discussions. 
	I am currently supported by Simons Foundation's Collaboration Grant for Mathematicians \#636937.}

\begin{abstract}
  Let $L_{l}=L(\la{sl}_{2l+1},-l-\frac{1}{2})$ be the simple vertex
  operator algebra based on the affine Lie algebra
  $\widehat{\la{sl}}_{2l+1}$ at boundary admissible level
  $-l-\frac{1}{2}$.

  We consider a lift $\nu$ of the Dynkin diagram involution of
  $A_{2l}=\la{sl}_{2l+1}$ to an involution of $L_{l}$.  The
  $\nu$-twisted $L_l$-modules are $A_{2l}^{(2)}$-modules of level
  $-l-\frac{1}{2}$ with an anti-homogeneous realization.  We classify
  simple $\nu$-twisted highest-weight (weak) $L_l$-modules using
  twisted Zhu algebras and singular vectors for
  $\widehat{\la{sl}}_{2l+1}$ at level $-l-\frac{1}{2}$ obtained by
  Per\v{s}e.

  We find that there are finitely many such modules up to isomorphism,
  and the $\nu$-twisted (weak) $L_l$-modules that are in category
  $\sO$ for $A_{2l}^{(2)}$ are semi-simple.
\end{abstract}
\maketitle

\section{Introduction}

In \cite{KacWak-modinv}, while studying the modular invariant
representations of affine Lie algebras, Kac and Wakimoto introduced
the notion of admissible highest-weight representations and classified these in 
\cite{KacWak-classification}.
Let $\la{g}$ be a finite dimensional simple Lie algebra of type $X_N$,
and let $\widehat{\la{g}}$ be the corresponding \emph{untwisted} affine Lie
algebra of type $X_N^{(1)}$ (see \cite[Table Aff 1]{Kac-book}).  Since
\cite{KacWak-modinv}, vertex operator algebras (say $L(\la{g},k)$)
based on the untwisted affine Lie algebra $\widehat{\la{g}}$ at
admissible levels $k$ have received a tremendous amount of attention.

In \cite{AdaMil-sl2}, Adamovi\'{c} and Milas analysed the case of
$\la{g}=\la{sl}_2$ for all admissible levels $k$, classified the weak
modules of $L(\la{sl}_2,k)$ that belong to category $\sO$ as
$\widehat{\la{sl}_2}$-modules and showed that this category is
semi-simple with finitely many equivalence classes of
irreducibles. They conjectured that this holds for all untwisted
affine Lie algebras.  The $\la{g}=\la{sl}_2$ case was also studied in
\cite{DongLiMas-sl2, FeiMal-sl2}.  In a celebrated achievement,
Arakawa proved this conjecture \cite{Ara-admaff}. Before
\cite{Ara-admaff}, several other specific cases of this conjecture
were known to be true, notably in type $C$ \cite{Ada-typeC}, in type
$A$ \cite{Per-typeA}, in type $B$ \cite{Per-typeB} and for $G_2$
\cite{AxtLee-G2}.  
In \cite{AdaMil-sl2, Ada-typeC,  Per-typeA, Per-typeB, AxtLee-G2},
the technique of Zhu algebras \cite{ZhuMod96, FreZhu} and an explicit 
knowledge of the singular vectors at the prescribed levels was used.

It is also important to consider categories larger than $\sO$, namely,
the categories generated by \emph{relaxed} highest weight modules.
For $\la{g}=\la{sl}_2$, simple relaxed highest-weight modules at
admissible levels were classified using the Zhu technology in
\cite{AdaMil-sl2}. Recently, a classification for arbitrary rank based
on Mathieu's coherent famililes \cite{Mat-class} is presented in
\cite{KawRid-II}. We will not pursue this direction here.

Kac-Wakimoto's work \cite{KacWak-modinv} also included a discussion of
affine Lie super-algebras, and indeed models related to
$\la{osp}(1|2)$ at admissible levels have been analysed in
\cite{SnaRidWood-osp, CreKanLiuRid-osp, Wood-osp}. However, it is not
clear if semi-simplicity holds beyond $\la{osp}(1|2n)$\footnote{I
  thank David Ridout for pointing this out.}, and \cite{Ara-admaff}
does not encompass affine super-algebras.

Despite all these stellar advances, the case of \emph{twisted} affine
Lie algebras (see \cite[Table Aff 2, Aff 3]{Kac-book}) has received
little to no attention. The most natural way to access modules for
$X_N^{(r)}$ where $r=2,3$ is by considering $\nu$-\emph{twisted}
modules for the VOAs based on the corresponding \emph{untwisted}
affine Lie algebras $X_N^{(1)}$ \cite{FreLepMeu, Li-twisted}.  Here,
$\nu$ is a lift of the non-trivial Dynkin diagram automorphism to
$\la{g}$ and $\nu$ fixes the chosen Cartan sub-algebra. $\nu$ is then
extended to act on the whole VOA.  We may modify $\nu$ by composing it
with $\exp(2\pi\im\cdot \ad h)$ for certain Cartan elements $h$ with
$\nu(h)=h$ \cite[Eq.\ 8.1.2]{Kac-book}.  This way, we get different
\emph{realizations} of $X_N^{(r)}$ differing primarily in their
gradings.

In this paper, we consider the case of $A_{2l}^{(2)}$ at level
$-l-\frac{1}{2}$ for $l\in\ZZ_{>0}$. We use the
\emph{anti-homogeneous} realization of $A_{2l}^{(2)}$ obtained from an
involutive lift $\nu$ of the Dynkin diagram automorphism of
$\la{g}=\la{sl}_{2l+1}=A_{2l}$. Here, anti-homogeneous refers to the
fact that our picture is exactly the opposite of the traditional one
-- our \emph{affine}, i.e., $0$th node for $A_{2l}^{(2)}$ is what is
usually the \emph{last}, i.e., $l$th node in the affine Dynkin
diagram, and our horizontal subalgebra is thus $\la{so}_{2l+1}=B_l$
and not $\la{sp}_{2l}=C_l$.

We use twisted Zhu algebras \cite{DongLiMas-twZhu} (see also
\cite{Yang-twaff}) and the singular vectors for
$\widehat{\la{sl}}_{2l+1}$ at level $-l-\frac{1}{2}$ obtained by
Per\v{s}e in \cite{Per-typeA}. Somewhat surprisingly, we find that the
top spaces of the $A_{2l}^{(2)}$ modules (which are naturally modules
for our horizontal subalgebra, $B_l$) are exactly the same as the top
spaces for the highest-weight $L(B_l,-l+\frac{3}{2})$-modules found in
\cite{Per-typeB}.  Letting $h^\vee$ denote the dual Coxeter number
\cite[Ch.\ 6]{Kac-book}, the relation between these levels for $l>1$
is that
\begin{align}
-l-\frac{1}{2} + h^\vee_{{A_{2l}}^{(2)}} = l+\frac{1}{2} = -l+\frac{3}{2} + h^\vee_{{B_{l}}^{(1)}}.
\end{align}
Our proof of admissibility of the $A_{2l}^{(2)}$ highest weights thus
obtained also uses a large portion of the corresponding proof in
\cite{Per-typeB}. The proof of semi-simplicity then proceeds as in
\cite{AdaMil-sl2,Ada-typeC, Per-typeA, Per-typeB, AxtLee-G2}, etc.,
with appropriate changes to accommodate twisted modules.

We find that there are \emph{two} inequivalent $\nu$-twisted
irreducible modules for $L(\ourla,\ourlev)$ with finite dimensional
top spaces (Remark \ref{rem:ordsimples}). Recall that $-l-\frac{1}{2}$
is a boundary admissible level \cite{KacWak-boundary} for
$A_{2l}^{(1)}$ and correspondingly, there is exactly one (up to
equivalence) irreducible with finite dimensional top space in the
untwisted sector \cite{Per-typeA}.

This naturally leads to the following speculations and considerations
that we are currently investigating.
\begin{enumerate}[topsep=0pt]
\item Perhaps the most important speculation we have is that the
  Adamovi\'{c}-Milas conjecture / Arakawa's theorem is true for
  \emph{twisted} affine Lie algebras as well. To be precise, we
  speculate that given a twisted affine Lie algebra $X_N^{(r)}$, and
  an admissible level $k$ for (the untwisted) $X_N^{(1)}$, there
  exists an appropriate realization of $X_N^{(r)}$ and a corresponding
  lift $\nu$ of the (non-trivial) diagram automorphism of $X_N$, such
  that $\nu$-twisted (weak) $L(X_N,k)$-modules which are in category
  $\sO$ as $X_N^{(r)}$-modules form a semi-simple category with
  finitely many irreducibles.
	
%	\item Add nilpotent orbits
	
\item In \cite{CreHuaYang-admaffvtc}, it was proved that the ordinary
  modules for $L(X_N,k)$ ($k$ is admissible level for the untwisted
  affine Lie algebra $X_N^{(1)}$), form a vertex tensor category; the
  rigidity of this category for the simply-laced cases was proved in
  \cite{Cre-admADErigid}.  Our results imply that in general, the
  category consisting of untwisted and $g$-twisted ordinary modules
  for $g\in\langle\nu\rangle$ will not be closed under twisted fusion.

  In our present case, the untwisted and $\nu$-twisted ordinary
  modules form semi-simple categories, but the untwisted sector has
  one simple (up to equivalences) and the $\nu$-twisted one has two
  inequivalent simples. The aforementioned closure under twisted fusion 
  is now forbidden by elementary considerations of tensor categories.

  In general, such ordinary $g$-twisted modules are integrable in the
  direction of $\la{g}^0$ (the fixed point subalgebra of $\la{g}=X_N$
  under $\nu$), thus it is natural to expect their (twisted) fusion to
  be integrable with respect to $\la{g}^0$, but it need not be
  $\la{g}$-integrable.
	
  It will be difficult but interesting to work out the
  \emph{twisted} fusion for our $\nu$-twisted modules, and perhaps
  also the fusion for the \emph{untwisted} modules for the
  corresponding orbifold. Here, the structure of this orbifold
  \cite{Ali-afforb} will be important to first classify its modules.
	
\item It will be very interesting to also analyse \emph{twisted}
  quantum Drinfeld-Sokolov reductions \cite{KacWak-twQDS} of the
  $\nu$-twisted modules we have found for appropriate nilpotents
  $f\in\la{sl}_{2l+1}$ fixed by $\nu$, and compare these to twisted
  representations of the corresponding $\sW$-algebras. Here again, one
  may take a slightly different route and investigate the relation of
  the structure and representation theory of the affine orbifold with
  that of the $\sW$-algebra orbifold.
\end{enumerate}

\section{Twisted affine Lie algebra $A_{2l}^{(2)}$}

\subsection{Twisted affine Lie algebra $A_{2l}^{(2)}$, basics}
We will consider what we call the \emph{anti-homogeneous} realization
of $A_{2l}^{(2)}$ and recall basic facts from \cite{Kac-book,
  Carter-book}.  Consider the generalized Cartan matrices:
\begin{align}
\tilde{A}=\bordermatrix{ & 0 & 1 \cr 0 & 2 & -1 \cr  1 & -4 & 2}, 
\,(l=1),%\quad\,\,
\end{align}
\begin{align}
\tilde{A}=
\bordermatrix{ 
	& 0  & 1 &  2 & \cdots & l-2 & l-1 & l \cr
	0       & 2  & -1 &     &  &   &  &  \cr	
	1       & -2 &  2 & -1  &  &    &  &  \cr	
	2       &    & -1 & 2   & \ddots &    &  & \cr	
	\vdots  &    &    &  \ddots   & \ddots & \ddots   &  &  \cr	
	l-2     &    &    &     & \ddots & 2   & -1 &  \cr	
	l-1     &    &    &     &  & -1   & 2 & -1 \cr	
	l       &    &    &     &  &    & -2 & 2 \cr	
},\,(l>1).
\end{align}
We have the corresponding (affine) Dynkin diagrams:
\begin{align}
\begin{matrix}
\begin{tikzpicture}
\node[fill, draw, circle,label=below:\tzl{\alpha_0},label=above:\tzl{1},inner sep=1.25pt]at(0,0) (a0) {};
\node[draw, circle,label=below:\tzl{\alpha_{1}},label=above:\tzl{2},inner sep=1pt]at(1,0) (a1) {};
\draw (0.15,0.09)--(0.76,0.09);
\draw (0.15,0.03)--(0.8,0.03);
\draw (0.15,-0.03)--(0.8,-0.03);
\draw (0.15,-0.09)--(0.76,-0.09);
\draw (0.75,0.15) to[bend right] (0.85,0) to[bend right] (0.75,-0.15);
\end{tikzpicture}\\
\scriptstyle{l=1}
\end{matrix}
\quad\quad\quad\quad\qquad
\begin{matrix}
\begin{tikzpicture}
\node[fill, draw, circle,label=below:\tzl{\alpha_0},label=above:\tzl{1},inner sep=1.25pt]at(0,0) (a0) {};
\node[draw, circle,label=below:\tzl{\alpha_1},label=above:\tzl{2},inner sep=1pt]at(1,0) (a1) {};
\node[draw, circle,label=below:\tzl{\alpha_2},label=above:\tzl{2},inner sep=1pt]at(2,0) (a2) {};
\node[draw, circle,label=below:\tzl{\alpha_3},label=above:\tzl{2},inner sep=1pt]at(3,0) (a3) {};
\node[] at(4,0) (a4) {};
\node[] at(5,0) (a5) {};
\node[draw, circle,label=below:\tzl{\alpha_{l-2}},label=above:\tzl{2},inner sep=1pt]at(6,0) (alm2) {};
\node[draw, circle,label=below:\tzl{\alpha_{l-1}},label=above:\tzl{2},inner sep=1pt]at(7,0) (alm1) {};
\node[draw, circle,label=below:\tzl{\alpha_{l}},label=above:\tzl{2},inner sep=1pt]at(8,0) (al) {};
\draw[-implies,double equal sign distance] (0.15,0)--(0.85,0);
\draw (1.15,0) -- (1.85,0);
\draw (2.15,0) -- (2.85,0);
\draw (3.15,0) -- (3.85,0);
\draw[dotted] (a4)--(a5);
\draw (5.15,0) -- (5.85,0);
\draw (6.15,0) -- (6.85,0);
\draw[-implies,double equal sign distance] (7.15,0)--(7.85,0);
\end{tikzpicture} \label{eqn:ahomdynkin} \\
\scriptstyle{l>1}
\end{matrix}
\end{align}
Here, the $0^{th}$ node is considered to be the \emph{affine} node and
the horizontal subalgebra of $A_{2l}^{(2)}$ is of type
$B_{l}=\la{so}_{2l+1}$ (unlike the usual convention where it turns out
to be $C_l=\la{sp}_{2l}$):
\begin{align}
\begin{matrix}
\begin{tikzpicture}
\node[draw, circle,label=below:\tzl{\alpha_{1}},inner sep=1pt]at(1,0) (a1) {};
\end{tikzpicture}\\
\scriptstyle{l=1}
\end{matrix}
\quad\quad\quad\quad\qquad
\begin{matrix}
\begin{tikzpicture}
\node[draw, circle,label=below:\tzl{\alpha_1},inner sep=1pt]at(1,0) (a1) {};
\node[draw, circle,label=below:\tzl{\alpha_2},inner sep=1pt]at(2,0) (a2) {};
\node[draw, circle,label=below:\tzl{\alpha_3},inner sep=1pt]at(3,0) (a3) {};
\node[] at(4,0) (a4) {};
\node[] at(5,0) (a5) {};
\node[draw, circle,label=below:\tzl{\alpha_{l-2}},inner sep=1pt]at(6,0) (alm2) {};
\node[draw, circle,label=below:\tzl{\alpha_{l-1}},inner sep=1pt]at(7,0) (alm1) {};
\node[draw, circle,label=below:\tzl{\alpha_{l}},inner sep=1pt]at(8,0) (al) {};
\draw (1.15,0) -- (1.85,0);
\draw (2.15,0) -- (2.85,0);
\draw (3.15,0) -- (3.85,0);
\draw[dotted] (a4)--(a5);
\draw (5.15,0) -- (5.85,0);
\draw (6.15,0) -- (6.85,0);
\draw[-implies,double equal sign distance] (7.15,0)--(7.85,0);
\end{tikzpicture} \label{eqn:ahomB}\\
\scriptstyle{l>1}
\end{matrix}
\end{align}
We have:
\begin{align}
a=(a_0,\dots,a_l)^t=(1,2,\dots,2)^t, &\quad \tilde{A}a=0\label{eqn:deltacoeffs}\\
a^\vee=(a_0^\vee,\dots,a_l^\vee)=(2,2,\dots,2,1), &\quad a^\vee \tilde{A}=0. \label{eqn:ccoeffs}
\end{align}

We will often use the following indexing sets:
\begin{align}
I=\{1,\dots,l\}, \hI=\{0,\dots,l\}.
\end{align}

The twisted affine Lie algebra $A_{2l}^{(2)}$ has Kac-Moody generators
$h_i,e_i,f_i$ ($i\in\hI$) and $d$ satisfying the usual relations
\cite{Kac-book}.  We let the Cartan subalgebra $\la{H}$ be spanned by
$h_0,h_1,\dots,h_l,d$, The simple roots $\alpha_i$ ($i\in\hI$) are
elements of $\la{H}^*$.  We will sometimes denote the pairing between
$\la{H}^*$ and $\la{H}$ by $(\cdot,\cdot)$. This notation will be
overloaded below.  For $i,j\in\hI$, $k\in I$ we have:
\begin{align}
\alpha_i(h_j)=(\alpha_i,h_j)=\widetilde{A}_{ji}, \quad
\alpha_0(d)=(\alpha_0,d)=1,\quad \alpha_k(d)=(\alpha_k,d)=0.
\end{align}
The canonical central element $c\in\la{H}$ of $A_{2l}^{(2)}$ and the basic imaginary root $\delta$ are expressed as:
\begin{align}
c = \sum_{0\leq i\leq l}a_i^\vee h_i= 2h_0+\dots +2h_{l-1}+h_l,\quad
\delta = \sum_{0\leq i\leq l}a_i \alpha_i = \alpha_0+2\alpha_1\dots +2\alpha_l.
 \label{eqn:cdelta}
\end{align}
We choose $h_1,\dots,h_l,c,d$ as a standard basis for $\la{H}$.
We have:
\begin{align}
\delta(c)=0,\,\, \delta(h_1)=0,\,\,\dots,\,\,\delta(h_l)=0, \,\,\delta(d)=1.\label{eqref:deltaaction}
\end{align}
We also consider $\Lambda_0^c\in\la{H}^*$ such that:
\begin{align}
\Lambda_0^c(c)=1,\,\, \Lambda_0^c(h_1)=0,\,\,\dots,\,\,\Lambda_0^c(h_l)=0, \,\,\Lambda_0^c(d)=0.\label{eqref:l0action}
\end{align}
Note that $\Lambda_0^c$ is $\frac{1}{2}\Lambda_0$, where $\Lambda_0$
is the fundamental weight corresponding to the $0$th node.  It is easy
to see that $\alpha_1,\dots,\alpha_l,\delta,\Lambda_0^c$ form a basis
of $\la{H}^*$.  The standard symmetric (non-degenerate) bilinear form
on $\la{H}$ is given by:
\begin{align}
(h_i,h) = (\alpha_i,h)\cdot \frac{a_i}{a_i^\vee},\quad i\in\hI,\,h\in\la{H},\quad\quad (d,d)=0.
\end{align}
The non-degenerate map $(\cdot,\cdot)$ leads to a (linear) isomorphism
$\iota:\la{H}\rightarrow \la{H}^*$ with ($i\in\hI$):
\begin{align}
(\iota(h),h_1) = (h,h_1),\,\, \mathrm{for\,all}\,\,h,h_1\in\la{H},\\
\iota(h_i)=\frac{a_i}{a_i^\vee}\cdot\alpha_i,\,\,\iota(c)=\delta,\,\,\iota(d)=\Lambda_0^c.
\end{align}
We may thus get a (non-degenerate) symmetric bilinear form on
$\la{H}^*$ by transport of structure.  It satisfies ($i,j\in\hI$ and
$k\in I$):
\begin{align}
(\alpha_i,\alpha_j)=\widetilde{A}_{ij}\cdot\frac{a_i^\vee}{a_i},\quad
(\delta,\alpha_k)= (\delta,\delta)= (\Lambda_0^c,\alpha_k)=(\Lambda_0^c,\Lambda_0^c)=0,\quad
(\delta,\Lambda_0^c)=1.\label{eqn:innerprodH*}
\end{align}
The squared root lengths are therefore ($k=1,\dots,l-1$):
\begin{align}
(\alpha_0,\alpha_0)=4,\,\,(\alpha_k,\alpha_k)=2,\,\,(\alpha_l,\alpha_l)=1.
\end{align}

The root system of $A_{2l}^{(2)}$ depends on $l$. Let $l>1$.  The root
system of the horizontal subalgebra $B_l$ can be realized as (with
$k=1,\dots,l-1$ and $i,j\in I$):
\begin{align}
  \alpha_k=\epsilon_k-\epsilon_{k+1},\,\,\alpha_l=\epsilon_l,\quad\mathrm{where}\,\, (\epsilon_i,\epsilon_j)=\delta_{ij}.
\end{align}
We have:
\begin{align}
\Phi^\rl = \{ \pm \epsilon_i \pm \epsilon_j\, |\, 1\leq i<j\leq l\},
\quad  
\Phi^\rs = \{ \pm \epsilon_i\, |\, i=1,\dots,l\},
\label{eqn:horizroots}
\end{align}
and the real roots for $A_{2l}^{(2)}$ are \cite{Carter-book}:
\begin{align}
\widehat{\Delta}^{\re} 
&= \widehat{\Phi}^{\rl}\cup\widehat{\Phi}^{\ri}\cup\widehat{\Phi}^{\rs}\nonumber\\
&=\{ 2\alpha_s + (2m+1)\delta\, |\, \alpha\in\Phi_s, m\in \ZZ\}\cup
\{ \alpha + m\delta\, |\, \alpha\in\Phi_l, m\in \ZZ\}\cup
\{ \alpha + m\delta\, |\, \alpha\in\Phi_s, m\in \ZZ\}
\label{eqn:rerootsl>1}
\end{align}
where the squared norms of roots in the respective sets are $4,2$ and
$1$.  With $k=1,\dots,l-1$, the fundamental weights of the horizontal
subalgebra are:
\begin{align}
\omega_k=\epsilon_1 + \cdots + \epsilon_k,\quad  \omega_l=\frac{1}{2}(\omega_1+\cdots+\omega_l).
\label{eqn:horizweights}
\end{align}
For $l=1$, the horizontal subalgebra is $\la{sl}_2$ with simple
positive root $\alpha_1$, and we have:
\begin{align}
\widehat{\Delta}^\re&= \widehat{\Phi}^{\rl}\cup \widehat{\Phi}^{\rs}=
\{\pm 2\alpha_1+(2m+1)\delta\,|\,m\in\ZZ\}\cup\{\pm\alpha_1+m\delta\,|\,m\in\ZZ\}\label{eqn:rerootsl=1}.
\end{align}
Here, note that $(\alpha_1,\alpha_1)=1$, and thus squared norms of the
roots in these sets are $4$ and $1$, respectively.  The fundamental
weight for the horizontal algebra is $\omega_1 = \frac{1}{2}\alpha_1$.

Let $\rho$ be any element of $\la{H}^*$ satisfying $\rho(h_i)=1$ for
all $i\in\hI$. We may take it to be:
$\rho = h^\vee \Lambda_0^c + \overline{\rho}$ where $\overline{\rho}$
is half the sum of positive roots of the horizontal sub-algebra and
$h^\vee=a_0^\vee+\cdots+a_l^\vee = 2l+1$ is the dual Coxeter number of
$A_{2l}^{(2)}$.  For $l=1$, $\overline{\rho}=\frac{1}{2}\alpha_1$.  If
$l>1$, we have:
\begin{align}
  \overline{\rho}= \left(l-\frac{1}{2}\right)\epsilon_1 + \left(l-\frac{3}{2}\right)\epsilon_2 + \cdots + \frac{1}{2}\epsilon_l
\label{eqn:rhol>1}.
\end{align}

Finally, recall the notion of Weyl group $W$ generated by reflections
$r_i$ ($i\in\hI$) satisfying $r_i(h)=h-(\alpha_i,h)h_i$ for
$h\in\la{H}$ and we transfer the action to $\la{H}^*$ by $\iota$.  We
have $W\cdot\{\alpha_0,\dots,\alpha_l\}=\widehat{\Delta}^\re$ and we
define $\widehat{\Delta}^{\vee,\re}=W\cdot\{h_0,\dots,h_l\}$, which is
the set of real coroots.  There is thus a bijection from real roots to
real coroots denoted by $^\vee$ such that
$\alpha_i\mapsto \alpha_i^\vee = h_i$, and it is not hard to prove,
using the invariance of $(\cdot,\cdot)$ under the Weyl group that for
$\lambda\in\la{H}^*,\alpha\in\widehat{\Delta}^\re$ that
\begin{align}
(\lambda,\alpha^\vee)=\frac{2}{(\alpha,\alpha)}(\lambda,\alpha).\label{eqn:rootcorooteqn}
\end{align}
Given $\lambda\in\la{H}^*$, define 
\begin{align}
\widehat{\Delta}_{\lambda}^{\vee,\re}=\{ \alpha^\vee \in \widehat{\Delta}^{\vee,\re}\,|\,(\lambda,\alpha^\vee)\in\ZZ\},
\quad \widehat{\Delta}_\lambda^{\re}=\{\alpha \in \widehat{\Delta}^{\re}\,|\,\alpha^\vee\in \widehat{\Delta}_{\lambda}^{\vee,\re}\}
\label{eqn:lambdaintegralcoroots}.
\end{align}
\begin{defi}\label{def:admissible}\cite{KacWak-modinv}	
  We say an element $\lambda\in\la{H}^*$ is an admissible weight if:
\begin{enumerate}
\item $(\lambda+\rho,\alpha^\vee)\not\in\{0,-1,-2,\cdots\}$ for all
  $\alpha^\vee\in \widehat{\Delta}^{\vee,\re}_+$ and
\item
  $\QQ\widehat{\Delta}_{\lambda}^{\vee,\re} = \QQ\{ h_0,\dots,h_l\}$.
\end{enumerate}
\end{defi}
\begin{rem}\label{rem:adm2}
  The second condition can be equivalently replaced with
  $\QQ\widehat{\Delta}_{\lambda}^{\re} = \QQ\{
  \alpha_0,\dots,\alpha_l\}$. %We will do so below. %\sk{check}
\end{rem}

\subsection{Twisted affinizations of Lie algebras}

Suppose we are given a finite dimensional (simple) Lie algebra
$\la{g}$ with a symmetric invariant bilinear form
$\langle \cdot,\cdot\rangle$.  Let $\nu$ be an automorphism of
$(\la{g},\langle\cdot,\cdot\rangle)$ of a finite order, say $T$.
Corresponding to $\nu$, we have the eigen-decomposition
\begin{align}
\la{g}=\bigoplus_{j\in \ZZ/T\ZZ}\la{g}^j,\quad x\in\la{g}^j\Leftrightarrow\nu(x)=e^{2\pi\im j/T }x.
\end{align}
Consider the affinization:
\begin{align}
\widehat{\la{g}}^{\frac{1}{T}\ZZ} = \la{g}\otimes \CC[t^{1/T},t^{-1/T}] \oplus \CC c.
\end{align}
We will often drop the superscript $^{\frac{1}{T}\ZZ}$ since it will
be clear from the context.  The element $c$ is central and the other
brackets are ($a,b\in\la{g}$, $m,n\in\frac{1}{T}\ZZ$):
\begin{align}
[a\otimes t^m,b\otimes t^n]=[a,b]\otimes t^{m+n} + m\delta_{m+n,0}\langle a,b\rangle c.
\end{align}
Define $\nu(t^{j/T}) = e^{-2\pi\im j/T}t^{j/T}$ and extend linearly to
$\CC[t^{1/T},t^{-1/T}]$. Also let $\nu(c)=c$.  We are interested in
the fixed point sub-algebra
\begin{align}
\widehat{\la{g}}[\nu]= \bigoplus_{j\in \ZZ/T\ZZ} \left(\la{g}^j\otimes t^{j/T}\CC[t,t^{-1}]\right) \oplus \CC c.
\end{align}
We shall obtain $A_{2l}^{(2)}$ via such twisted affinization
 of $\la{g}=A_{2l}=\la{sl}_{2l+1}$.

\subsection{Anti-homogeneous realization of $A_{2l}^{(2)}$}\label{sec:ahom}
We start by fixing some notation.  Fix $l\in\ZZ_{>0}$.  Consider
$\la{gl}_{2l+1}$ spanned by elementary matrices $E_{i,j}$ (or simply
$E_{ij}$) with $1$ in row $i$ and column $j$, zeros everywhere else.
Let $E_i=E_{i,i+1}$, $F_{i}=E_{i+1,i}$, $H_i=E_{i,i}-E_{i+1,i+1}$ be
the standard choices of simple root vectors and simple coroots for
$\la{g}=\la{sl}_{2l+1}\subset \la{gl}_{2l+1}$.  Let
$\la{g}=\la{n}_-\oplus\la{h}\oplus \la{n}_+$ be the triangular
decomposition of $\la{g}$.  Let
$E_{\theta}=E_{1,2l+1}=[\cdots[[E_1,E_2],E_3],\cdots,E_{2l}]$.

The anti-homogeneous realization is achieved via an involutive lift of
the diagram automorphism of $\la{sl}_{2l+1}$ which we now describe.

Define $\nu(E_{i,j})=-(-1)^{i-j}E_{2l+2-j,2l+2-i}$.  It is
straightforward to prove that $\nu$ is an involution of
$\la{gl}_{2l+1}$ and also of the Lie subalgebra
$\la{g}=\la{sl}_{2l+1}$.  Corresponding to $\nu$ we have the
decomposition $\la{g}=\la{g}^0\oplus\la{g}^1$ (where the superscripts
are understood as elements of $\ZZ_2$).  It is clear that
$\nu(H_i)=H_{2l+1-i}$, for $i=1,\dots, 2l$, and thus $\nu$ is an
involutive lift of the Dynkin diagram automorphism of $\la{g}$.
Observe that $E_{\theta}\in\la{g}^1$.

The fixed points $\la{g}^0$ form a simple Lie algebra of type
$B_l=\la{so}_{2l+1}$ with the following Chevalley generators
\cite{Carter-book}.
\begin{align}
e_i = E_i + E_{2l+1-i},\,\,\, f_i =F_i + F_{2l+1-i},\,\,\, h_i = H_i + H_{2l+1-i},&\quad (i=1,\dots l-1),\nonumber\\
\overline{e_l} = \sqrt{2}(E_l+E_{l+1}),\,\,\, 
\overline{f_l} = \sqrt{2}(F_l+F_{l+1}),\,\,\, 
\overline{h_l} = 2(H_l+H_{l+1}).&\label{eqn:Blgens}
\end{align}
For convenience, we denote:
\begin{align}
{e_l} = E_l+E_{l+1},\,\,\, 
{f_l} = F_l+F_{l+1},\,\,\, 
{h_l} = H_l+H_{l+1}.&\label{eqn:elhlfl}
\end{align}
Note again that the actual generators for $B_l$ (which will also get
promoted to a subset of generators for $A_{2l}^{(2)}$ below) are
indeed $e_1,\dots,e_{l-1},\overline{e_l}$,
$h_1,\dots,h_{l-1},\overline{h_l}$,
$f_1,\dots,f_{l-1},\overline{f_l}$.  We have introduced $e_l,f_l,h_l$
only to save ourselves from keeping track of the various scalars.

Given any $a\in\la{g}$, we let 
\begin{align}
a=a^++a^-,\quad a^+=\frac{1}{2}(a+\nu a)\in\la{g}^0,\,\, a^-=\frac{1}{2}(a-\nu a)\in\la{g}^1\label{eqn:pmnotation}.
\end{align}

We let $\la{g}^0=\la{n}_-^0\oplus\la{h}^0\oplus \la{n}_+^0$ be the
triangular decompositions with respect to our choices of root vectors.
Note that $\la{n}_+^0$ is spanned by $E_{i,j}^+$ for
$1\leq i<j\leq 2l+1$ and that $\dim(\la{h}^0)=l$.

Later, we shall require the dimension of weight $0$ space of
$\la{g}^1$ as a $\la{g}^0$-module. One may calculate this directly by
decomposing $\la{g}$ with respect to $\la{g}^0$. Here we present one
more approach.  Temporarily, let $L(\omega)$ denote irreducible
$\la{g}^0\cong \la{so}_{2l+1}$ module with highest weight $\omega$.
As a $\la{g}^0$-module, $\la{g}^1\cong L(2\omega_1)$ and is generated
by the highest weight vector $E_{\theta}$. Further,
$\mathrm{Sym}^2L(\omega_1) \cong L(2\omega_1) \oplus \CC$ and
$L(\omega_1)$ is the defining representation of dimension $2l+1$.  It
can be seen that if $\omega$ is a weight of $L(\omega_1)$ then so is
$-\omega$, $0$ is a weight, and every weight space is one
dimensional. Thus, the $0$ weight space of $\mathrm{Sym}^2L(\omega_1)$
has dimension $l+1$, and the $0$ weight space of
$\la{g}^1\cong L(2\omega_1)$ has dimension $l$.

Now, $\widehat{\la{sl}}_{2l+1}[\nu]$ gives us an anti-homogeneous
realization of $A_{2l}^{(2)}$. Considering the numbering from
\eqref{eqn:ahomdynkin}, we let the Kac-Moody generators to be the ones
given in \eqref{eqn:Blgens} for $i=1,\dots ,l$.  As for $h_0,e_0,f_0$,
we take them to be:
\begin{align}
h_0 &= -H_{\theta} + \frac{1}{2}c = -(H_1+\cdots +H_{2l}) + \frac{1}{2}c,\,\,
e_0 = E_{2l+1,1}\otimes t^{1/2},\,\,
f_0 = E_{1,2l+1}\otimes t^{-1/2}.
\end{align}

The involution $\nu$ extends to the universal enveloping algebra
$\cU(\la{g})$ and we have
$\cU(\la{g}^0)\subsetneq\cU(\la{g})^0\subsetneq \cU(\la{g})$.  Later,
we will be interested in certain two-sided ideals
$I\subset\cU(\la{g}^0)$.

\begin{rem}
  It is possible to achieve this realization of $A_{2l}^{(2)}$ by
  using the Chevalley involution \cite[Eq.\ 1.3.4]{Kac-book} of
  $A_{2l}$. However, it is convenient to use an automorphism that
  respects the triangular decomposition of $\la{g}$.
\end{rem}

\section{Twisted Zhu algebra: Preliminaries}

Let $(V,Y,\vac,\omega)$ be a vertex operator algebra \cite{LepLi} and
let $g$ be an automorphism of finite order $T$ of $V$.  Let $V^{j}$
($j=0,\dots, T-1$) be the subspace of eigenvalue $e^{2\pi \im j/T}$
for $g$.  Following \cite{DongLiMas-twZhu}, we now define the twisted
Zhu algebra $A_g(V)$ as follows.  Let $u\in V^{j}$ ($0\leq j< T$) be
$L(0)$- and $g$- homogeneous element and let $v\in V$.  Define
\begin{align}
u\circ_gv&=\res_x\left( \frac{(1+x)^{\mathrm{wt}u-1+\delta_j+\frac{j}{T}}}{x^{1+\delta_j}}Y(u,x)v\right),\label{eqn:circg}\\
u\ast_g v&=
\begin{cases}
\res_x\left(\dfrac{(1+x)^{\mathrm{wt}u}}{x}Y(u,x)v\right)&\quad\,\mathrm{if}\,\,j= 0\\
0&\quad\,\mathrm{if}\,\,\mathrm{otherwise}.\label{eqn:zhug}\\
\end{cases}
\end{align}
where we take $\delta_j=1$ when $j=0$ and $\delta_j=0$ if $j\neq 0$.
Extend $\circ_g,\ast_g$ to $V$ linearly.  Further define
\begin{align}
O_g(V)=\mathrm{Span}\{u\circ_g v\,\vert\,u,v\in V\},\quad 
A_g(V)=V/O_g(V).
\end{align}
Taking $v=\vac$ in \eqref{eqn:circg} immediately gives us:
\begin{align}
V^i \subset O_g(V)\,\,\mathrm{if}\,\,i\not\equiv 0\pmod{T}\label{eqn:oddzerozhu}.
\end{align}

We will denote the image in $A_g(V)$ of $v\in V$ by $\lz v\rz_g$.  It
was shown in \cite{DongLiMas-twZhu} that $O_g(V)$ is a two-sided ideal
with respect to $\ast_g$ and that $A_g(V)$ is an associative algebra
under product $\ast_g$ with $\lz\vac\rz$ as the unit and
$\lz\omega\rz$ belonging to the center.  When $g=1$,
$\circ_g,\ast_g,O_g(V),A_g(V)$ are simply denoted as
$\circ,\ast,O(V),A(V)$, respectively.  We recall the following basic
theorems (and their twisted analogues) from
\cite{ZhuMod96,FreZhu,DongLiMas-twZhu,Yang-twaff}.
\begin{thm}
  We have:
  \begin{enumerate}
  \item\cite{FreZhu,Yang-twaff} Let $I$ be a $g$-stable ideal of $V$,
    and suppose $\vac\not\in I$, $\omega\not\in I$. Then, the image of
    $I$ in $A_g(V)$, denoted as $A_g(I)$ is a two-sided
    ideal. Moreover, $A_g(V/I)\cong A_g(V)/A_g(I)$.
  \item\cite[Thm.\ 7.2]{DongLiMas-twZhu} There is a bijective
    correspondance between the set of equivalence classes of simple
    $A_g(V)$ modules and {weak, $\frac{1}{T}\ZZ$-gradable} $g$-twisted
    $V$-modules (see \cite[Def.\ 3.3]{DongLiMas-twZhu}, where these
    modules are called admissible, not to be confused with
    \cite{KacWak-modinv}).
  \end{enumerate}
\end{thm}

\begin{rem}
	The first part of the theorem above is proved for $g=1$ (the 
	untwisted case) in \cite{FreZhu}. It is not hard to extend 
	the proof to general $g$ \cite{Yang-twaff}.
\end{rem}

Now, for the rest of the section, let $V=V(\la{g},k)$ be the
(universal) Verma module vertex operator algebra based on
$(\la{g},\langle\cdot,\cdot\rangle)$ with level $k\neq -h^{\vee}$
\cite{LepLi}.  Let $g$ be an automorphism of $V$ order $T\neq 1$
lifted from an automorphism $g$ of
$(\la{g},\langle\cdot,\cdot\rangle)$ of the same order $T$.
\begin{thm}
We have:
\begin{enumerate}
\item\cite{Yang-twaff} There exists an (explicit) isomorphism of
  associative algebras
  $F:A_g(V(\la{g},k))\xrightarrow{\cong}\la{U}(\la{g}^0)$.
\item\cite{FreZhu} Let $x\in\la{g}^0$ and $v\in V$. Then,
  \begin{align}
    F(\lz x(0) v\rz_g) = [ x, \lz v\rz_g], \label{eqn:0modebrac}
  \end{align}
  where both sides are zero if
  $v\in V^{1}\oplus \cdots \oplus V^{T-1}$.
\item\cite{FreZhu} Let $x_1,x_2,\dots,x_m\in\la{g}^0$,
  $n_1,n_2,\dots,n_m\in\ZZ_{\geq 0}$. Then under the isomorphism
  above,
  \begin{align}
    F(\lz x_1(-n_1-1)x_2(-n_2-1)\cdots x_m(-n_m-1)\vac\rz_g) = (-1)^{n_1+n_2+\dots+n_m}x_mx_{m-1}\cdots x_1.\label{eqn:zhuanti2}
  \end{align}
\item The previous part immediately implies that for $x\in \la{g}^0$,
  $n\in\ZZ_{\geq 0}$ and $v\in V$, we have:
  \begin{align}
    F( \lz x(-n-1) v\rz_g) = (-1)^nF(\lz v\rz_g)x.\label{eqn:zhuanti1}
  \end{align}
\end{enumerate}
Henceforth, we will suppress $F(\cdots)$ and simply identify
$A_g(V(\la{g},k))$ with $\la{U}(\la{g}^0)$.
\end{thm}

\begin{defi}
  We have an action ${}_L$ of $\la{g}^0$ on $\cU(\la{g}^0)$ given by
  $x_Lu=[x,u]$ where $x\in\la{g}^0$, $u\in\cU(\la{g}^0)$.  We may and
  do extend the action ${}_L$ of $\la{g}^0$ to an action of
  $\cU(\la{g}^0)$.
\end{defi}

\begin{thm}
  Suppose that the (unique) maximal $\widehat{\la{g}}$-submodule
  $J(\la{g},k)$ of $V(\la{g},k)$ is generated by a single
  $g$-homogeneous singular vector $v$.  Let $\cU(\la{g})v$ be the
  $\la{g}$-module generated by $v$ where $x\in\la{g}$ acts on $v$ by
  $x(0)$.  Let
  \begin{align}
    \cR = \lz \cU(\la{g})v\rz_g
    =\lz\, (\cU(\la{g})v)\cap V(\la{g},k)^{0}\,\rz_g
    =\lz\, (\cU(\la{g})v)^0 \,\rz_g .
  \end{align}
  We have the following.
  \begin{enumerate}
  \item $\cR$ is a finite-dimensional module for $\cU(\la{g}^0)$ under
    the ${}_L$ action.
  \item Let $L(\la{g},k)=V(\la{g},k)/J(\la{g},k)$ be the unique simple
    quotient of $V(\la{g},k)$. Then,
    \begin{align}
      A_g(L(\la{g},k))=\frac{\cU(\la{g}^0)}{ \langle \cR \rangle}
      \label{eqn:AgLisUg0quot}
    \end{align}
    where $\langle \cR \rangle$ denotes two sided ideal of
    $\cU(\la{g}^0)$ generated by $\cR$.
  \item A $\la{g}^0$-module $M$ is a $A_g(L(\la{g},k))$-module iff
    $\langle \cR \rangle\cdot M=0$.
	\end{enumerate}
\end{thm}
\begin{proof}
  This theorem is analogous to the corresponding theorems in the
  untwisted setup, and in the twisted setting, our proof is very
  similar to the proof of \cite[Thm.\ 6.3]{Yang-twaff}.
	
  All elements of $\cR$ have the same conformal weight as that of $v$,
  and each conformal weight space of $V(\la{g},k)$ is
  finite-dimensional, hence $\cR$ is finite-dimensional. The fact that
  $\cR$ is closed under ${}_L$ action is immediate from
  \eqref{eqn:0modebrac}.

  For the second assertion, it is enough prove that
  $\lz J(\la{g},k) \rz_g = \langle \cR\rangle$.  Observe that if $X$
  is a subspace of $\cU(\la{g}^0)$ that is closed under the $_L$
  action and also under the right multiplication by $\cU(\la{g}^0)$
  then, $X$ is a two-sided ideal.  Indeed, for $a\in\la{g}^0$ and
  $x\in X$, we have $ax=[a,x]-xa$, and both terms on the right-hand
  side belong to $X$, giving us the closure of $X$ under the
  left-action.  In light of \eqref{eqn:0modebrac},
  $\lz J(\la{g},k)\rz_g$ is closed under ${}_L$ and
  \eqref{eqn:zhuanti1} implies that it is also closed under the right
  action of $\cU(\la{g}^0)$. Thus, it is a two sided ideal. Clearly,
  $\cR\subset \lz J(\la{g},k)\rz_g$, and thus
  $\langle\cR\rangle\subset \lz J(\la{g},k)\rz_g$.

  For the reverse inclusion, note that $J(\la{g},k)$ is spanned by
  terms of the sort
  \begin{align}
    y=a_1(-n_1-1)a_2(-n_2-1)\cdots a_t(-n_t-1)x,
  \end{align}
  where $a_i\in\la{g}$ are $g$-homogeneous and are arranged so that
  all $a_i$'s in $\la{g}^0$ are to the right, $n_i\in\ZZ_{\geq 0}$ and
  $x$ is a $g$-homogeneous element of $\cU(\la{g})v$.  We proceed by
  induction on $t$. The case for $t=0$ is clear: $\lz x\rz_g\in \cR$.
  Now let $t>0$.  If all $a_i$'s and $x$ are already fixed by $g$ then
  \eqref{eqn:zhuanti1} immediately tells us that
  $\lz y\rz_g\in \langle \cR\rangle$.  Suppose that $y$ is fixed by
  $g$ (otherwise its projection is $0$ anyway) and that $a_1$ is in
  $\la{g}^r$, $1\leq r\leq T-1$.  We have the following relation
  \cite{DongLiMas-twZhu}:
  \begin{align}
    \res_x\frac{(1+x)^{r/T}}{x^{m+1}}Y(a_1(-1)\vac,x)v = a_1(-m-1)v + \frac{r}{T}a_1(-m)v + \cdots  \in O_g(V(\la{g},k))
  \end{align}
  for all $v\in V(\la{g},k)$ and $m\geq 0$. Repeating this relation
  for $m=n_1, n_1-1,\dots$, it is clear that for some scalar $\alpha$,
  \begin{align}
    y&\equiv_{O_g(V(\la{g},k))} \alpha\cdot a_1(0)a_2(-n_2-1)\cdots a_t(-n_t-1)x + \mathrm{shorter\,terms}\\
     &\equiv_{O_g(V(\la{g},k))} \alpha \cdot a_2(-n_2-1)\cdots a_t(-n_t-1)a_1(0)x + \mathrm{shorter\,terms}.
  \end{align}
  We may similarly peel off all the elements $a_2, \cdots, a_j$ which
  are not in $\la{g}^0$ and put them near $x$.  We thus see, for some
  scalar $\alpha'$:
  \begin{align}
    y&\equiv_{O_g(V(\la{g},k))} \alpha'\cdot a_{j+1}(-n_2-1)\cdots a_t(-n_t-1)\cdot a_j(0)\cdots a_1(0)x + \mathrm{shorter\,terms}.
  \end{align}
  Since $y$ and $a_{j+1}\cdots a_t$ are all fixed by $g$,
  $a_j(0)\cdots a_2(0)a_1(0)x \in(\cU(\la{g})v)^0$. Now, again,
  \eqref{eqn:zhuanti1} and induction hypothesis give us that
  $\lz y\rz_g\in\langle \cR\rangle$.
\end{proof}

Now we recall a couple of important results that form the basis of all
our calculations.

\begin{defi}
  Recall that $\cR$ is a $\la{g}^0$ module under the $_L$ action. We
  have already chosen a Cartan subalgebra for $\la{g}^0$, namely
  $\la{h}^0$.  Let $\cR_0$ be the weight $0$ subspace of $\cR$ with
  respect $\la{h}^0$.
\end{defi}

\begin{thm}(\cite[Lem.\ 3.4.3]{AdaMil-sl2}, \cite[Prop.\
  13]{Per-typeB}) Let $L(\lambda)$ be an irreducible highest-weight
  $\la{g}^0$-module with highest-weight $\lambda$ and a highest-weight
  vector $v_\lambda$. The following statements are equivalent.
  \begin{enumerate}
  \item $L(\lambda)$ is an $A_g(L(\la{g},k))$-module.
  \item $\cR\cdot L(\lambda)=0$.
  \item $\cR_0\cdot v_\lambda=0$.
  \end{enumerate}
\end{thm}

\begin{defi}
  In the notation of the previous theorem, for every $r\in\cR_0$ there
  exists a (unique) polynomial $p_r\in \la{S}(\la{h}^0)$ such that
  $rv_{\lambda}=p_r(\lambda)v_\lambda$.  Define
  $\cP_0=\{ p_r\,|\, r\in\cR_0\}$. % \sk{check}
\end{defi}	

We immediately have:
\begin{cor}\label{cor:onetoonepolymod}(\cite[Cor.\ 2.10]{Per-typeA})
  There is a one-to-one correspondence between:
  \begin{enumerate}
  \item Irreducible, highest-weight $A_g(L(\la{g},k))$ modules and
  \item weights $\lambda\in(\la{h}^0)^\ast$ such that $p(\lambda)=0$
    for all $p\in\cP_0$.
  \end{enumerate}
\end{cor}

We now present some calculations that will be used below. Let
$a\in\la{g}^j$, $0< j< T$, $b\in\la{g}$. Then,
\begin{align}
  &(a(-1)\vac)\circ_g(b(-1)\vac)\nonumber\\
  &=\res_x\left(
    \dfrac{(1+x)^{j/T}}{x}
    (\cdots + a(-1)b(-1)\vac x^0 + a(0)b(-1)\vac x^{-1}+  a(1)b(-1)\vac x^{-2})
    \right)\nonumber\\
  &=\res_x\left(
    \left(\sum_{n\geq 0}{j/T\choose n}x^{n-1}\right)
    (\cdots + a(-1)b(-1)\vac x^0 + a(0)b(-1)\vac x^{-1}+  a(1)b(-1)\vac x^{-2})
    \right)\nonumber\\
  &=a(-1)b(-1)\vac + \frac{j}{T}[a,b](-1)\vac +\frac{j(j-T)}{2T^2}k\langle a,b\rangle\vac.
\end{align}
This implies that for $a\in\la{g}^j$, ($0<j<T$) and $b\in \la{g}$,
\begin{align}
  a(-1)b(-1)\vac \equiv_{O_g(V)} -\frac{j}{T}[a,b](-1)\vac -\frac{j(j-T)}{2T^2}k\langle a,b\rangle\vac.
\end{align}
Or, equivalently,
\begin{align}
  \lz a(-1)b(-1)\vac\rz_g = -\frac{j}{T}\lz[a,b](-1)\vac\rz_g -\frac{j(j-T)}{2T^2}k\langle a,b\rangle\lz\vac\rz_g.
\label{eqn:proja-1b-1}
\end{align}
Since $\langle\cdot,\cdot\rangle$ is $g$-invariant, both sides are
zero if $b^{(T-j)}=0$.%\sk{check}

For general elements $a,b\in\la{g}$, we have:
\begin{align}
  a(-1)&b(-1)\vac=(a^{(0)}+\cdots+ a^{(T-1)})(-1)(b^{(0)}+\cdots+ b^{(T-1)})(-1)\vac\nonumber\\
       &=\left(a^{(0)}(-1)b^{(0)}(-1)\vac + a^{(1)}(-1)b^{(T-1)}(-1)\vac +\cdots+
         a^{(T-1)}(-1)b^{(1)}(-1)\vac\right) +\dots
\end{align}	
where the last ellipses denote terms that are in
$V^{(1)}\oplus\cdots \oplus V^{(T-1)}$.  So, using
\eqref{eqn:oddzerozhu} and \eqref{eqn:proja-1b-1}
\begin{align}
  &\lz a(-1)b(-1)\vac\rz_g\nonumber\\
  &	=\lz a^{(0)}(-1)b^{(0)}(-1)\vac\rz_g
    -\sum_{0<j<T}\frac{j}{T}\lz[a^{(j)},b^{(T-j)}](-1)\vac\rz_g
    -\sum_{0<j<T}\frac{j(j-T)}{2T^2}k\langle a^{(j)},b^{(T-j)}\rangle\lz\vac\rz_g
    .\label{eqn:projab}
\end{align}	
Henceforth, we will drop the subscript $_g$.

\section{$\nu$-Twisted Zhu algebra for $L(\ourla,\ourlev)$}

Fix $l\in\ZZ_{>0}$ and let $\la{g}=\la{sl}_{2l+1}$ as before and let
$k=\ourlev$.  Recall that $V(\la{g},k)$ is the (universal) generalized
Verma module VOA and $J(\la{g},k)$ is its (unique) maximal proper
ideal.
\begin{thm}
  From \cite{Per-typeA} we have:
  \begin{enumerate}[leftmargin=*]
  \item The vector
    \begin{align}
      {v}=\sum_{i=1}^{2l} \frac{2l-2i+1}{2l+1}E_{\theta}(-1)H_i(-1)\vac
      +\sum_{i=1}^{2l-1}E_{1,i+1}(-1)E_{i+1,2l+1}(-1)\vac
      -\frac{1}{2}(2l-1)E_\theta(-2)\vac
    \end{align}
    is a singular vector in $V(\la{g},k)$.
  \item The ideal $J(\la{g},k)$ is generated by $v$, that is,
    $J(\la{g},k)=\cU(\widehat{\la{g}})v$.
  \end{enumerate}
\end{thm}
\begin{proof}
  Our notation is slightly different from \cite{Per-typeA}.  Negative
  of the singular vector given in \cite{Per-typeA} is:
  \begin{align}
    -\sum_{i=1}^{2l} \frac{2l-2i+1}{2l+1} H_i(-1)e_\theta(-1)\vac
    + \sum_{i=1}^{2l-1}e_{\epsilon_1-\epsilon_{i+1}}(-1)e_{\epsilon_{i+1}-\epsilon_{2l+1}}(-1)\vac 
    + \frac{1}{2}(2l-1)e_\theta(-2)\vac,
  \end{align}
  where they define for $i<j$
  \begin{align}
    e_{\epsilon_i-\epsilon_j}= [E_{j-1},[E_{j-2},[\cdots[E_{i+1},E_i]\cdots]]],
    \quad e_{\theta}= e_{\epsilon_1-\epsilon_{2l+1}}.
  \end{align}
  It can be seen that
  \begin{align}
    e_{\epsilon_i-\epsilon_j}=-(-1)^{i-j}E_{i,j},\quad e_{\theta}=-E_{\theta}.
  \end{align}
  We thus get
  \begin{align}
    \sum_{i=1}^{2l} \frac{2l-2i+1}{2l+1} H_i(-1)E_\theta(-1)\vac 
    + \sum_{i=1}^{2l-1}E_{1,i+1}(-1)E_{i+1,2l+1}(-1)\vac
    - \frac{1}{2}(2l-1)E_\theta(-2)\vac.
  \end{align}
  In the first summation, $[H_i(-1),E_{\theta}(-1)]=0$ if $1<i<2l$ and
  $[H_i(-1),E_{\theta}(-1)]=E_{\theta}(-2)$ if $i=1,2l$. We thus get
  the required formula.
\end{proof}

\begin{lem}\label{lem:viseven}
  We have $\nu(v)=v$.
\end{lem}
\begin{proof}
  We have:
  \begin{align}
    &\nu(v)\nonumber\\
    &=\sum_{i=1}^{2l} -\frac{2l-2i+1}{2l+1}E_{\theta}(-1)H_{2l+1-i}(-1)\vac
      +\sum_{i=1}^{2l-1}E_{2l-i+1,2l+1}(-1)E_{1,2l-i+1}(-1)\vac
      +\frac{2l-1}{2}E_\theta(-2)\vac
      \nonumber\\
    &=\sum_{i=1}^{2l} \frac{2l-2i+1}{2l+1}E_{\theta}(-1)H_{i}(-1)\vac
      +\sum_{i=1}^{2l-1}E_{i+1,2l+1}(-1)E_{1,i+1}(-1)\vac
      +\frac{2l-1}{2}E_\theta(-2)\vac\nonumber\\
    &=\sum_{i=1}^{2l} \frac{2l-2i+1}{2l+1}E_{\theta}(-1)H_{i}(-1)\vac
      +\sum_{i=1}^{2l-1}\left(E_{1,i+1}(-1)E_{i+1,2l+1}(-1)-E_{1,2l+1}(-2)\right)\vac
      +\frac{2l-1}{2}E_\theta(-2)\vac\nonumber\\
    &=v,
  \end{align}
  where the first equality follows by definition of $\nu$ and the
  second by re-indexing the summations.
\end{proof}

Our next task is to calculate enough information about
$\cR=\lz \cU(\la{g})v\rz$ so that we can use Corollary
\ref{cor:onetoonepolymod}.  The $\la{g}$-weight of $v$ is $\theta$,
and as $\la{g}$-module, $\cU(\la{g})v$ is isomorphic to the adjoint
module of $\la{g}$ with $v\mapsto E_\theta$. As $\la{g}^0$-modules, we
then have $\cU(\la{g})v\cong \la{g}^0\oplus\la{g}^1$. Note that
$E_\theta\in\la{g}^1$ and so $\cU(\la{g}^0)v\cong \la{g}^1$ as
$\la{g}^0$-modules. Since $v$ is $\nu$-fixed, we have
$\cR =\lz \cU(\la{g})v\rz=\lz \cU(\la{g}^0)v\rz$.  From Section
\ref{sec:ahom} we know that $\dim(\cR_0)=\dim((\la{g}^1)_0)=l$ and
thus we seek $l$ independent polynomials in $\cP_0$.

\begin{lem}
  The projection of $v$ on the twisted Zhu algebra is given by the
  following formula:
  \begin{align}
    \lz v \rz =\sum_{i=1}^{2l-1}E^+_{i+1,2l+1} E_{1,i+1}^+.
  \end{align}
\end{lem}
\begin{proof}
  First, it is easy to see that:
  $$\sum_{i=1}^{2l}
  \frac{2l-2i+1}{2l+1}E_{\theta}(-1)H_i(-1)\vac-\frac{1}{2}(2l-1)E_\theta(-2)\vac\in
  V(\la{g},\ourlev)^1.$$ Using \eqref{eqn:projab}, we get:
  \begin{align}
    &\sum_{i=1}^{2l-1}\lz E_{1,i+1}(-1)E_{i+1,2l+1}(-1)\vac\rz\\
    &=\sum_{i=1}^{2l-1}\lz E_{1,i+1}^+(-1)E_{i+1,2l+1}^+(-1)\vac\rz
      -\frac{1}{2}\lz [E_{1,i+1}^-,E_{i+1,2l+1}^-](-1)\vac\rz
      -\frac{l+\frac{1}{2}}{8}\langle E_{1,i+1}^-,E_{i+1,2l+1}^-\rangle\lz \vac\rz\\
    &=\sum_{i=1}^{2l-1}\lz E_{1,i+1}^+(-1)E_{i+1,2l+1}^+(-1)\vac\rz
      -\frac{1}{2}\lz [E_{1,i+1}^-,E_{i+1,2l+1}^-](-1)\vac\rz.
  \end{align}
  For $i=1,\dots,2l-1$, we have:
  \begin{align} [E_{1,i+1}^-,E_{i+1,2l+1}^-]&=\frac{1}{4}
    [E_{1,i+1}+(-1)^iE_{2l+1-i,2l+1},E_{i+1,2l+1}+(-1)^{i}E_{1,2l+1-i}]
    =0.
  \end{align}
  Using \eqref{eqn:zhuanti2} for the first term, we get the required
  result.
\end{proof}

\begin{lem}
  Consider $E_{l+1,1}-(-1)^lE_{2l+1,l+1}\in\la{g}^0$.  Let
  \begin{align}
    v_1= 2(E_{l+1,1}-(-1)^lE_{2l+1,l+1})_L\lz v\rz \in\cU(\la{g}^0).
    \label{eqn:v1}
  \end{align}
  Then, we have:
  \begin{align}
    v_1=&(-1)^l\sum_{1\leq i<l} (E_{1,i+1}-(-1)^iE_{2l+1-i,2l+1})(E_{i+1,l+1} -(-1)^{l-i}E_{l+1,2l+1-i})\label{eqn:v1.1}\\
        &+(-1)^l(E_{1,1}-E_{2l+1,2l+1} ) (E_{1,l+1}-(-1)^lE_{l+1,2l+1})\label{eqn:v1.2}\\
        &-\frac{(-1)^l}{2}(E_{1,l+1}-(-1)^lE_{l+1,2l+1})\label{eqn:v1.3}\\
        &+\sum_{l<i\leq 2l-1}(-1)^l(E_{i+1,l+1} -(-1)^{l-i}E_{l+1,2l+1-i})(E_{1,i+1}-(-1)^iE_{2l+1-i,2l+1})\label{eqn:v1.4}.
  \end{align}
\end{lem}
\begin{proof}
  \begin{align}
    v_1&=2(E_{l+1,1}-(-1)^lE_{2l+1,l+1})_L\lz v\rz \nonumber\\
       &=2\sum_{i=1}^{2l-1}[E_{l+1,1}-(-1)^lE_{2l+1,l+1},E^+_{i+1,2l+1}] E_{1,i+1}^+
         +2\sum_{i=1}^{2l-1}E^+_{i+1,2l+1} [E_{l+1,1}-(-1)^lE_{2l+1,l+1}, E_{1,i+1}^+]\nonumber\\
       &=\sum_{i=1}^{2l-1}[E_{l+1,1}-(-1)^lE_{2l+1,l+1},\, E_{i+1,2l+1}-(-1)^iE_{1,2l+1-i}] E_{1,i+1}^+\nonumber\\
       &\quad+\sum_{i=1}^{2l-1}E^+_{i+1,2l+1} [E_{l+1,1}-(-1)^lE_{2l+1,l+1}, \, E_{1,i+1}-(-1)^iE_{2l+1-i,2l+1}]\nonumber\\
       &=\sum_{i=1}^{2l-1}(-1)^l(-\delta_{i,l}E_{2l+1,2l+1}+\delta_{i,l}E_{1,1} +E_{i+1,l+1} -(-1)^{l-i}E_{l+1,2l+1-i})  E_{1,i+1}^+\nonumber\\
       &\quad+\sum_{i=1}^{2l-1}E^+_{i+1,2l+1} (E_{l+1,i+1}-(-1)^{l-i}E_{2l+1-i,l+1}-\delta_{i,l}E_{1,1}+\delta_{i,l}E_{2l+1,2l+1}).
  \end{align}
  Our aim is to convert the expressions so that elements from
  $\la{n}_+^0$ are to the right.  Note that $\la{n}_+^0$ is spanned by
  $E_{i,j}^+$ for $1\leq i<j\leq 2l+1$.  We split both summations into
  $i<l,i=l,i>l$ parts.  In the first summation, all terms are already
  in this form, but we still rewrite them with a view towards future
  calculations.  The $i<l$ component is:
  \begin{align}
    &\frac{1}{2}\sum_{1\leq i<l}(-1)^l(E_{i+1,l+1} -(-1)^{l-i}E_{l+1,2l+1-i})  (E_{1,i+1}-(-1)^iE_{2l+1-i,2l+1})\nonumber\\
    &=\frac{(-1)^l}{2}\sum_{1\leq i<l}
      \left( (E_{1,i+1}-(-1)^iE_{2l+1-i,2l+1})
      (E_{i+1,l+1} -(-1)^{l-i}E_{l+1,2l+1-i}) 
      -E_{1,l+1}+(-1)^lE_{l+1,2l+1}\right)\nonumber\\
    &=\frac{(-1)^l}{2}\sum_{1\leq i<l}
      (E_{1,i+1}-(-1)^iE_{2l+1-i,2l+1})
      (E_{i+1,l+1} -(-1)^{l-i}E_{l+1,2l+1-i}) \\
    &\quad -\frac{(-1)^l(l-1)}{2}(E_{1,l+1}-(-1)^lE_{l+1,2l+1}).
  \end{align}
  The $i=l$ term is:
  \begin{align}
    \frac{(-1)^l}{2}(E_{1,1}-E_{2l+1,2l+1})(E_{1,l+1}-(-1)^lE_{l+1,2l+1}).
  \end{align}
  We keep the $i>l$ terms unchanged:
  \begin{align}
    \frac{1}{2}\sum_{l<i\leq 2l-1}(-1)^l(E_{i+1,l+1} -(-1)^{l-i}E_{l+1,2l+1-i}) (E_{1,i+1}-(-1)^iE_{2l+1-i,2l+1}).
  \end{align}
  For the second summation, the $i<l$ terms become:
  \begin{align}
    \sum_{1\leq i<l}&E^+_{i+1,2l+1} (E_{l+1,i+1}-(-1)^{l-i}E_{2l+i-1,l+1})\\
    =&\frac{1}{2}\sum_{1\leq i<l}(E_{i+1,2l+1}-(-1)^iE_{1,2l+1-i}) (E_{l+1,i+1}-(-1)^{l-i}E_{2l+1-i,l+1})\nonumber\\
    =&\frac{1}{2}\sum_{1\leq i<l}\left((E_{l+1,i+1}-(-1)^{l-i}E_{2l+1-i,l+1})(E_{i+1,2l+1}-(-1)^iE_{1,2l+1-i})
       -E_{l+1,2l+1}+(-1)^lE_{1,l+1}\right) \nonumber\\
    =&\frac{1}{2}\sum_{l< i\leq 2l-1}(E_{l+1,2l+1-i}-(-1)^{l-i}E_{i+1,l+1})(E_{2l+1-i,2l+1}-(-1)^iE_{1,i+1})
       -E_{l+1,2l+1}+(-1)^lE_{1,l+1} \nonumber\\
    =&\frac{1}{2}\left(\sum_{l< i\leq 2l-1}(-1)^l(E_{i+1,l+1}-(-1)^{l-i}E_{l+1,2l+1-i})(E_{1,i+1}-(-1)^iE_{2l+1-i,2l+1})\right)\nonumber\\
                    &+\frac{(-1)^l(l-1)}{2}(E_{1,l+1}-(-1)^lE_{l+1,2l+1}).
  \end{align}
  The $i=l$ term is:
  \begin{align}
    E^+_{l+1,2l+1} &(-E_{1,1}+E_{2l+1,2l+1})\nonumber\\
                   &=\frac{1}{2}
                     (E_{l+1,2l+1}-(-1)^lE_{1,l+1}) (-E_{1,1}+E_{2l+1,2l+1})\nonumber\\
                   &=\frac{1}{2}
                     (-E_{1,1}+E_{2l+1,2l+1})(E_{l+1,2l+1}-(-1)^lE_{1,l+1}) +
                     \frac{1}{2}(E_{l+1,2l+1}-(-1)^lE_{1,l+1})\nonumber\\
                   &=\frac{(-1)^l}{2}
                     (E_{1,1}-E_{2l+1,2l+1})(E_{1,l+1}-(-1)^lE_{l+1,2l+1}) +
                     \frac{1}{2}(E_{l+1,2l+1}-(-1)^lE_{1,l+1})\nonumber.
  \end{align}
  The $i>l$ term is:
  \begin{align}
    \sum_{l<i\leq 2l-1}&E^+_{i+1,2l+1} (E_{l+1,i+1}-(-1)^{l-i}E_{2l+1-i,l+1})\\
                       &=\frac{1}{2}\sum_{l<i\leq 2l-1}(E_{i+1,2l+1}-(-1)^iE_{1,2l+1-i}) (E_{l+1,i+1}-(-1)^{l-i}E_{2l+1-i,l+1})\\
                       &=\frac{1}{2}\sum_{1\leq i<l} (E_{2l+1-i,2l+1}-(-1)^iE_{1,i+1}) (E_{l+1,2l+1-i}-(-1)^{l-i}E_{i+1,l+1})\\
                       &=\frac{(-1)^l}{2}\sum_{1\leq i<l} (E_{1,i+1}-(-1)^iE_{2l+1-i,2l+1}) (E_{i+1,l+1}-(-1)^{l-i}E_{l+1,2l+1-i})
  \end{align}
  Combining everything, we get the required formula for $v_1$.
\end{proof}

Now we shall get many elements in the weight zero space $\cR_0$.

\begin{thm}\label{thm:ord2polys}
  Let $1\leq j\leq l$. Recall \eqref{eqn:Blgens} and
  \eqref{eqn:elhlfl}.  We have:
  \begin{align}
    -(-1)^j(f_jf_{j-1}\cdots f_1\,f_{j+1}f_{j+2}\cdots f_{l})_L v_1
    &=h_j\left(h_j + 2\sum_{j<i\leq l}h_i + (l-j) - \frac{1}{2}\right)+ \modnz.
      \label{eqn:fjv1}
  \end{align}
\end{thm}
\begin{proof}
  It is not hard to see that for every $1\leq j\leq l$,
  $(f_jf_{j-1}\cdots f_1\,f_{j+1}f_{j+2}\cdots f_{l})_L v_1\in\cR_0$.
	
  Throughout this proof, it will be often beneficial for us to do the
  calculations in $\cU(\la{g})$ or $\cU(\la{g})^0$. Since we are sure
  that the final answer is to be in $\cU(\la{g}^0)$, we will carefully
  omit the terms not in $\cU(\la{g}^0)$ that appear in the
  intermediate steps. Recall that $\la{n}_+^0$ is spanned by
  $E_{i,j}^+$ for $1\leq i<j\leq 2l+1$.

  The calculation corresponding to the term \eqref{eqn:v1.1} is the
  longest and we break it into several steps.

  First, we consider the term $E_{1,i+1}E_{i+1,l+1}$. Let
  $1\leq i<j\leq l$.  We have:
  \begin{align}
    (f_j&f_{j-1}\cdots f_{i+1}\,f_{i}\cdots f_1\,f_{j+1}f_{j+2}\cdots f_{l})_L
          (E_{1,i+1}E_{i+1,l+1})\nonumber\\
        &=(F_jF_{j-1}\cdots F_{i+1}\,F_{i}\cdots F_1\,F_{j+1}F_{j+2}\cdots F_{l})_L
          (E_{1,i+1}E_{i+1,l+1})\nonumber\\
        &=(-1)^{l-j}(F_jF_{j-1}\cdots F_{i+1}\,F_{i}\cdots F_1)_L
          (E_{1,i+1}E_{i+1,j+1})\nonumber\\
        &=-(-1)^{l-j}(F_jF_{j-1}\cdots F_{i+1})_L
          (H_iE_{i+1,j+1})\nonumber\\
        &=(-1)^{l-j}H_iH_j+\modn+\label{eqn:top_i<j}
  \end{align}
  Now let $i=j$, but note that we only allow $1\leq i<l$ in
  \eqref{eqn:v1.1}.
  \begin{align}
    (f_j&f_{j-1}\cdots f_1\,f_{j+1}f_{j+2}\cdots f_{l})_L
          (E_{1,j+1}E_{j+1,l+1})\nonumber\\
        &=(F_jF_{j-1}\cdots F_1\,F_{j+1}F_{j+2}\cdots F_{l})_L
          (E_{1,j+1}E_{j+1,l+1})\nonumber\\
        &=(-1)^{l-j}(F_jF_{j-1}\cdots F_1)_L
          (E_{1,j+1}H_{j+1})\nonumber\\
        &=(-1)^{l-j}(F_jF_{j-1}\cdots F_1)_L
          (H_{j+1}E_{1,j+1}+E_{1,j+1})\nonumber\\
        &=-(-1)^{l-j}(H_{j+1}H_j+H_j)+\modn.\label{eqn:top_i=j}
  \end{align}
  Now let $i>j$, again noting that we only allow $1\leq i<l$ in
  \eqref{eqn:v1.1}.
  \begin{align}
    (f_j&f_{j-1}\cdots f_1\,f_{j+1}\cdots f_{i}\,f_{i+1}\cdots f_{l})_L
          (E_{1,i+1}E_{i+1,l+1})\nonumber\\
        &=(F_jF_{j-1}\cdots F_1\,F_{j+1}\cdots F_{i}\,F_{i+1}\cdots  F_{l})_L
          (E_{1,i+1}E_{i+1,l+1})\nonumber\\
        &=(-1)^{l-i}(F_jF_{j-1}\cdots F_1\,F_{j+1}\cdots F_{i})_L
          (E_{1,i+1}H_{i+1})\nonumber\\
        &=(-1)^{l-i}(F_jF_{j-1}\cdots F_1\,F_{j+1}\cdots F_{i})_L
          (H_{i+1}E_{1,i+1}+E_{1,i+1})\nonumber\\
        &=-(-1)^{l-j}(H_{i+1}H_j+H_j) + \modn.\label{eqn:top_i>j}
  \end{align}
  Note that if we place $i=j$ in \eqref{eqn:top_i>j}, we get
  \eqref{eqn:top_i=j}, thus we combine these two equations.  Combining
  \eqref{eqn:top_i<j}, \eqref{eqn:top_i=j}, \eqref{eqn:top_i>j} for a
  fixed $1\leq j\leq l$, we have:
  \begin{align}
    (&f_jf_{j-1}\cdots f_1\,f_{j+1}f_{j+2}\cdots  f_{l})_L\left((-1)^l\sum_{1\leq i<l}E_{1,i+1}E_{i+1,l+1}+(-1)^lE_{2l+1-i,2l+1}E_{l+1,2l+1-i}\right)\nonumber\\
     &=(-1)^l(f_jf_{j-1}\cdots f_1\,f_{j+1}f_{j+2}\cdots  f_{l})_L\left(\sum_{1\leq i<l}(1+\nu)(E_{1,i+1}E_{i+1,l+1})\right)\nonumber\\
     &=(-1)^l(1+\nu)\sum_{1\leq i<l}(f_jf_{j-1}\cdots f_1\,f_{j+1}f_{j+2}\cdots  f_{l})_L(E_{1,i+1}E_{i+1,l+1})\nonumber\\
     &=%(-1)^l(-1)^{l-j}
       (-1)^j
       (1+\nu)\left( \sum_{1\leq i<j}H_iH_j - \sum_{j\leq i<l}(H_{i+1}H_j+H_j) +\modn\right).\nonumber\\
     &=(-1)^{j}H_j\left( \sum_{1\leq i<j}H_i - \sum_{j\leq i<l}(H_{i+1}+1) \right)
       +(-1)^{j}H_{2l+1-j}\left( \sum_{1\leq i<j}H_{2l+1-i} - \sum_{j\leq i<l}(H_{2l-i}+1) \right)\nonumber\\
     &\quad\quad+\modn.\label{eqn:v1.1part1}
  \end{align}
  Now we consider the term $E_{1,i+1}E_{l+1,2l+1-i}$.  First let
  $1\leq i<j\leq l$.  We have:
  \begin{align}
    (f_j&\cdots f_{i+1}\,f_i\cdots f_1\,f_{j+1}\cdots f_{l-1} f_l)_L(E_{1,i+1}E_{l+1,2l+1-i})\nonumber\\
        &=(f_j\cdots f_{i+1}\,f_i\cdots f_1\,F_{2l-j}\cdots F_{l+2}F_{l+1})_L(E_{1,i+1}E_{l+1,2l+1-i})\nonumber\\
        &=(f_j\cdots f_{i+1}\,f_i\cdots f_1)_L(E_{1,i+1}E_{2l+1-j,2l+1-i})\nonumber\\
        &=(f_jf_{j-1}\cdots f_{i+1}\,F_i\cdots F_1)_L(E_{1,i+1}E_{2l+1-j,2l+1-i})\nonumber\\
        &=(f_jf_{j-1}\cdots f_{i+1})_L(-H_iE_{2l+1-j,2l+1-i})\nonumber\\
        &=(F_{2l+1-j}F_{2l+2-j}\cdots F_{2l-i})_L(-H_iE_{2l+1-j,2l+1-i})\nonumber\\
        &=-(-1)^{i-j}\,H_iH_{2l+1-j}+\modn.\label{eqn:mix_i<j}
  \end{align}
  Now let $i=j$, but note that $1\leq i<l$.
  \begin{align}
    (f_j&\cdots  f_1\,f_{j+1}\cdots  f_l)_L(E_{1,j+1}E_{l+1,2l+1-j})\nonumber\\
        &=(f_j\cdots  f_1\,F_{2l-j}F_{2l-j-1}\cdots F_{l+1})_L(E_{1,j+1}E_{l+1,2l+1-j})\nonumber\\
        &=(f_j\cdots  f_1)_L(-E_{1,j+1}H_{2l-j})\nonumber\\
        &=(f_j\cdots  f_1)_L(-H_{2l-j}E_{1,j+1})\nonumber\\
        &=(F_jF_{j-1}\cdots F_1)_L(-H_{2l-j}	E_{1,j+1})\nonumber\\
        &=H_{2l-j}H_j+\modn.\label{eqn:mix_i=j}
  \end{align}
  Now let $j<i$, but again note that $1\leq i<l$.
  \begin{align}
    (f_j&\cdots f_1\,f_{j+1}\cdots f_{i}\,f_{i+1}\cdots f_l)_L(E_{1,i+1}E_{l+1,2l+1-i})\nonumber\\
        &=(f_j\cdots f_1\,f_{j+1}\cdots f_{i}\,F_{2l-i}F_{2l-i-1}\cdots F_{l+1})_L(E_{1,i+1}E_{l+1,2l+1-i})\nonumber\\
        &=(f_j\cdots f_1\,f_{j+1}\cdots f_{i})_L(-E_{1,i+1}H_{2l-i})\nonumber\\
        &=(f_j\cdots f_1\,f_{j+1}\cdots f_{i})_L(-H_{2l-i}E_{1,i+1})\nonumber\\
        &=(-1)^{i-j}H_{2l-i}H_j+\modn\label{eqn:mix_i>j}.
  \end{align}
  Again, note that placing $i=j$ in \eqref{eqn:mix_i>j} gets us
  \eqref{eqn:mix_i=j}, thus we combine these two.  Combining
  \eqref{eqn:mix_i<j}, \eqref{eqn:mix_i=j}, \eqref{eqn:mix_i>j}, for a
  fixed $1\leq j\leq l$, we see:
  \begin{align}
    (&f_jf_{j-1}\cdots f_1\,f_{j+1}f_{j+2}\cdots f_l)_L
       \left((-1)^l \sum_{1\leq i<l}-(-1)^{l-i} E_{1,i+1}E_{l+1,2l+1-i}-(-1)^iE_{2l+1-i,2l+1}E_{i+1,l+1}\right)\nonumber\\
     &=(-1)^{i+1}(f_jf_{j-1}\cdots f_1\,f_{j+1}f_{j+2}\cdots f_l)_L
       \left( \sum_{1\leq i<l} E_{1,i+1}E_{l+1,2l+1-i}+(-1)^lE_{2l+1-i,2l+1}E_{i+1,l+1}\right)\nonumber\\
     &=(-1)^{i+1}(f_jf_{j-1}\cdots f_1\,f_{j+1}f_{j+2}\cdots f_l)_L
       \left( \sum_{1\leq i<l} (1+\nu)(E_{1,i+1}E_{l+1,2l+1-i})\right)\nonumber\\
     &=(-1)^{i+1}(1+\nu)
       \left( \sum_{1\leq i<l} (f_jf_{j-1}\cdots f_1\,f_{j+1}f_{j+2}\cdots f_l)_L(E_{1,i+1}E_{l+1,2l+1-i})\right)\nonumber\\
     &=(-1)^{i+1}(1+\nu)
       \left( \sum_{1\leq i<j}-(-1)^{i-j}H_iH_{2l+1-j} + \sum_{j\leq i<l}(-1)^{i-j}H_{2l-i}H_{j} +\cU(\la{g})\la{n}_+\right)\nonumber\\
     &=(-1)^{j}
       \left( \sum_{1\leq i<j}\left(H_iH_{2l+1-j} +H_{2l+1-i}H_{j}\right)
       - \sum_{j\leq i<l}\left(H_{2l-i}H_{j}+H_{i+1}H_{2l+1-j}\right) \right)+\cU(\la{g})\la{n}_+\nonumber\\
     &=(-1)^{j}H_j
       \left( \sum_{1\leq i<j}H_{2l+1-i}- \sum_{j\leq i<l}H_{2l-i} \right)+(-1)^{j}H_{2l+1-j}
       \left( \sum_{1\leq i<j}H_{i}- \sum_{j\leq i<l}H_{i+1} \right)
       +\modn.\label{eqn:v1.1part2}
  \end{align}
  Finally, we put together \eqref{eqn:v1.1part1} and
  \eqref{eqn:v1.1part2}:
  \begin{align}
    (&f_jf_{j-1}\cdots f_1\,f_{j+1}f_{j+2}\cdots f_l)_L\left(
       (-1)^l\sum_{1\leq i<l} (E_{1,i+1}-(-1)^iE_{2l+1-i,2l+1})(E_{i+1,l+1} -(-1)^{l-i}E_{l+1,2l+1-i})\right)\nonumber\\
     &=(-1)^{j}H_j\left( \sum_{1\leq i<j}H_i - \sum_{j\leq i<l}(H_{i+1}+1) \right)
       +(-1)^{j}H_{2l+1-j}\left( \sum_{1\leq i<j}H_{2l+1-i} - \sum_{j\leq i<l}(H_{2l+1-i}+1) \right)\nonumber\\
     &+(-1)^{j}H_j
       \left( \sum_{1\leq i<j}H_{2l+1-i}- \sum_{j\leq i<l}H_{2l-i} \right)+(-1)^{j}H_{2l+1-j}
       \left( \sum_{1\leq i<j}H_{i}- \sum_{j\leq i<l}H_{i+1} \right)
       +\modn\nonumber\\
     &=(-1)^{j}h_j\left( \sum_{1\leq i<j}h_i - \sum_{j<i\leq l}h_{i} -(l-j)\right)+\modn.\label{eqn:v1.1fin}
  \end{align}
  In fact, in the last equality, we may now replace $\modn$ with
  $\modnz$.

  For \eqref{eqn:v1.2} and \eqref{eqn:v1.3}, note that
  \begin{align}
    (&f_jf_{j-1}\cdots f_1\,f_{j+1}f_{j+2}\cdots f_{l})_L (E_{1,l+1}-(-1)^lE_{l+1,2l+1} )\nonumber\\
     &=(f_jf_{j-1}\cdots f_1\,f_{j+1}f_{j+2}\cdots f_{l})_L (E_{1,l+1}+\nu E_{1,l+1} )\nonumber\\
     &=(1+\nu)((f_jf_{j-1}\cdots f_1\,f_{j+1}f_{j+2}\cdots f_{l})_L E_{1,l+1})\nonumber\\
     &=(1+\nu)((F_jF_{j-1}\cdots F_1\,F_{j+1}F_{j+2}\cdots F_{l})_L E_{1,l+1})\nonumber\\
     &=(1+\nu)((-1)^{l+j+1}H_j)=(-1)^{l+j+1}h_j.
  \end{align}
  The effect of applying
  $(f_jf_{j-1}\cdots f_1\,f_{j+1}f_{j+2}\cdots f_{l})_L$ on
  \eqref{eqn:v1.2} is thus:
  \begin{align}
    (f_j&f_{j-1}\cdots f_1\,f_{j+1}f_{j+2}\cdots f_{l})_L \left((-1)^l(E_{1,1}-E_{2l+1,2l+1} ) (E_{1,l+1}-(-1)^lE_{l+1,2l+1})\right)\nonumber\\
        &=(-1)^{j+1}(E_{1,1}-E_{2l+1,2l+1} )h_j + \modnz\nonumber\\
        &=(-1)^{j+1}(h_1+\cdots+h_l)h_j + \modnz.
          \label{eqn:v1.2fin}
  \end{align}
  and for \eqref{eqn:v1.3} we get:
  \begin{align}
    (f_j&f_{j-1}\cdots f_1\,f_{j+1}f_{j+2}\cdots f_{l})_L \left(-\frac{(-1)^l}{2}(E_{1,l+1}-(-1)^lE_{l+1,2l+1})\right)
          =\frac{(-1)^j}{{2}}h_j+ \modnz\label{eqn:v1.3fin}.
  \end{align}

  It is not hard to see that
  $(f_jf_{j-1}\cdots f_1\,f_{j+1}f_{j+2}\cdots f_{l})_L$ applied to
  the terms \eqref{eqn:v1.4} will only yield terms in $\modnz$.

  Combining \eqref{eqn:v1.1fin}, \eqref{eqn:v1.2fin} and
  \eqref{eqn:v1.3fin} we get:
  \begin{align}
    (&f_jf_{j-1}\cdots f_1\,f_{j+1}f_{j+2}\cdots f_l)_Lv_1\nonumber\\
     &=(-1)^{j}h_j\left( \sum_{1\leq i<j}h_i - \sum_{j<i\leq l}h_{i} -(l-j)\right)-(-1)^{j}(h_1+\cdots+h_l)h_j +
       \frac{(-1)^j}{{2}}h_j+ \modnz\nonumber\\
     &=(-1)^{j}\,  h_j\left(-h_j - 2\sum_{j<i\leq l}h_i+ \frac{1}{2}- (l-j)\right) +  \modnz.
  \end{align}
\end{proof}

\begin{rem}\label{rem:importPerseB}
  There is a very illuminating way to write the polynomials in
  \eqref{eqn:fjv1}.  Let $1\leq j<l$. Note that the coroot
  $h_{\epsilon_j+\epsilon_{j+1}}$ is
  $h_j + 2h_{j+1}+\cdots +2h_{l-1}+\overline{h_l}$ which is the same
  as $h_j + 2\sum_{j<i\leq l} h_i$.  In the case $j=l$, we write
  $h_l = \frac{1}{2}\overline{h_l}$. All in all, we see that we have
  got the following polynomials:
  \begin{align}
    p_j &= h_j\left(h_{\epsilon_j + \epsilon_{j+1}}+l-j+\frac{1}{2}\right)\quad \mathrm{for}\,\,1\leq j\leq l-1,\\
    p_l &= \frac{1}{4}\overline{h_l}(\overline{h_l}-1).
  \end{align}
  Observe that $p_i\in \cP_0$ and they are linearly independent. Thus
  $\dim(\cP_0)\geq l$. However, since $\dim(\cP_0)\leq \dim(\cR_0)=l$,
  we in fact have an equality and hence the $p_i$ span $\cP_0$.

  These are exactly the polynomials (up to a factor of $4$ in $p_{l}$)
  obtained by Per\v{s}e \cite{Per-typeB} in relation to the top spaces
  of $B_l$ modules at level $-l+\frac{3}{2}$. Thus, we may immediately
  import relevant results from \cite{Per-typeB} on zero sets of these
  polynomials.
\end{rem}

\begin{thm}\cite[Prop.\ 30]{Per-typeB}\label{thm:htwt}
  For every subset
  $S=\{i_1,i_2,\cdots,i_k\}\subseteq \{1,2,\cdots,l-1\}$ with
  $i_1<\cdots<i_k$, define:
  \begin{align}
    \mu_S&=\sum_{j=1}^k\left(i_j+2\sum_{s=j+1}^k(-1)^{s-j}i_s + (-1)^{k-j+1}\left(l-\frac{1}{2}\right)\right)\omega_{i_j},\label{eqn:wtmu}\\
    \mu_S'&=\omega_l+\sum_{j=1}^k\left(i_j+2\sum_{s=j+1}^k(-1)^{s-j}i_s + (-1)^{k-j+1}\left(l+\frac{1}{2}\right)\right)\omega_{i_j}.\label{eqn:wtmu'}
  \end{align}
  Then, $\{\mu_S,\mu_S'\,|\,S\subset\{1,2,\dots,l-1\}\}$ provides the
  complete list of highest weights of irreducible highest-weight
  $A_{\nu}(L(\ourla,\ourlev))$-modules.
\end{thm}

\begin{rem}\label{rem:ordsimples}
  In \eqref{eqn:wtmu} and \eqref{eqn:wtmu'}, notice that the
  coefficient of each of the $\omega_{i_j}$ $(j=1,\dots,k)$ is an
  element of $\frac{1}{2}+\ZZ$. This means that we have obtained
  exactly two weights that are dominant integral for $B_l$. These
  correspond to $S$ being the empty set:
  $\mu_\phi=0, \mu_{\phi}'=\omega_l$. These are precisely the highest
  weights of the simple ordinary (i.e., Virasoro mode $L(0)$ acts
  semisimply with finite dimensional weight spaces, and weights are
  bounded from below) $\nu$-twisted modules.
\end{rem}

\section{Admissibility and complete reducibility}

\subsection{Admissibility}
Due to the results in \cite{Li-twisted}, every (weak) $\nu$-twisted
$L(\ourla,\ourlev)$-module is naturally a module for the twisted
affine Lie algebra $A_{2l}^{(2)}$ of level $\ourlev$.  As weights for
$A_{2l}^{(2)}$, the weights obtained in Theorem \ref{thm:htwt} become:
\begin{align}
  \lambda_S = \left(-l-\frac{1}{2}\right)\Lambda_0^c+\mu_S,\quad \lambda_S' = \left(-l-\frac{1}{2}\right)\Lambda_0^c+\mu_S'.
\end{align}
We now prove that these are admissible, see Definition
\ref{def:admissible}.
\begin{thm}
  For every $S\subseteq\{1,2,\dots,l-1\}$, the weights $\lambda_S$ and
  $\lambda_S'$ are admissible for $A_{(2l)}^{(2)}$.
\end{thm}	
\begin{proof}
  Recall that $\rho = \overline{\rho} + h^\vee\Lambda_0^c$ and observe
  that
\begin{align}
  \lambda_S+\rho=\left(l+\frac{1}{2}\right)\Lambda_0^c+\overline{\rho}+\mu_S,
  \quad 
  \lambda_S'+\rho=\left(l+\frac{1}{2}\right)\Lambda_0^c+\overline{\rho}+\mu_S'. 
\end{align}
First, let $l=1$. Then, the only choice for $S$ is the empty set
$\phi$, and we have two weights, $\mu_\phi=0$,
$\mu_\phi'=\omega_1=\frac{1}{2}{\alpha_1}$. Let $\mu$ be one of these,
and let $\lambda$ be $-\frac{3}{2}\Lambda_0^c+\mu$.

Consider $m\in\ZZ$,
$\widetilde{\alpha}=\pm\alpha_1+m\delta\in\widehat{\Phi}^\rs_+$.  If
$m>0$, then, recalling \eqref{eqn:rootcorooteqn},
\eqref{eqn:innerprodH*},
\begin{align}
  (\lambda+\rho,\widetilde{\alpha}^\vee)=\frac{2}{1}\left(\frac{3}{2}\Lambda_0^c+\frac{1}{2}\alpha_1+\mu,\pm\alpha_1+m\delta\right)
  =3m\pm 1\pm 2(\mu,\alpha_1)>0
\end{align}
since $(\mu,\alpha_1)=0$ or $\frac{1}{2}$. If $m=0$, then,
$\widetilde{\alpha}=\alpha_1$, and
\begin{align}
  (\lambda+\rho,\widetilde{\alpha}^\vee)=\frac{2}{1}\left(\frac{3}{2}\Lambda_0^c+\frac{1}{2}\alpha_1+\mu,\alpha_1\right)
  = 1+ 2(\mu,\alpha_1)>0.
\end{align}
Now let $m\in\ZZ$,
$\widetilde{\alpha}=\pm
2\alpha_1+(2m+1)\delta\in\widehat{\Phi}^\rl_+$.  Necessarily,
$m\geq 0$ and, recalling \eqref{eqn:rootcorooteqn},
\eqref{eqn:innerprodH*},
\begin{align}
  (\lambda+\rho,\widetilde{\alpha}^\vee)=\frac{2}{4}\left(\frac{3}{2}\Lambda_0^c+\frac{1}{2}\alpha_1+\mu,\pm 2\alpha_1+(2m+1)	\delta\right)
  =\frac{3}{4}(2m+1)\pm \frac{1}{2}\pm(\mu,\alpha_1)\not\in\ZZ,
\end{align}
since $(\mu,\alpha_1)=0$ or $\frac{1}{2}$.  Thus the first condition
of admissibility is satisfied.

For the second condition, note that
$\alpha_1,\delta-\alpha_1\in\widehat{\Phi}^\rs$.  We have
$\alpha_1,\delta-\alpha_1\in\widehat{\Delta}^{\re}_{\lambda_\phi}$
since:
\begin{align*}
  (\lambda_\phi,(\alpha_1)^\vee)=2\left(-\frac{3}{2}\Lambda_0^c,\alpha_1\right)=0,\quad 
  (\lambda_\phi,(\delta-\alpha_1)^\vee)=2\left(-\frac{3}{2}\Lambda_0^c,\delta-\alpha_1\right)=-3.
\end{align*}
We have
$\alpha_1,\delta-\alpha_1\in\widehat{\Delta}^{\re}_{\lambda_\phi'}$
since:
\begin{align*}
  (\lambda_\phi',(\alpha_1)^\vee)=2\left(-\frac{3}{2}\Lambda_0^c+\frac{1}{2}\alpha_1,\alpha_1\right)=1,\quad 
  (\lambda_\phi',(\delta-\alpha_1)^\vee)=2\left(-\frac{3}{2}\Lambda_0^c+\frac{1}{2}\alpha_1,\delta-\alpha_1\right)=-4.
\end{align*}

Now, let $l>1$. Most of the work for this case has been already done
in \cite[Lem.\ 32]{Per-typeB}.  Also, as in \cite{Per-typeB}, the
proof for $\lambda_S$ and $\lambda_S'$ is similar, so we only present
the former.

Suppose that $\alpha\in\Phi^{\rs}\cup\Phi^{\rl}$, $m\in\ZZ$ such that
$\widetilde{\alpha}=\alpha+m\delta\in\widehat{\Phi}_{+}^{\rs}\cup\widehat{\Phi}_{+}^{\ri}$.
Then, recalling \eqref{eqn:rootcorooteqn}, \eqref{eqn:innerprodH*}, we
get the following, exactly as in \cite[Eq.\ 12]{Per-typeB}:
\begin{align}
  ( \lambda_S+\rho,\widetilde{\alpha}^\vee)=
  \left(\left(l+\frac{1}{2}\right)\Lambda_0^c +\overline{\rho}+\mu_S,{(\alpha+m\delta)}^\vee\right)=
  \frac{2}{(\alpha,\alpha)}\left(m\left(l+\frac{1}{2}\right)+(\overline{\rho},\alpha)+(\mu_S,\alpha)\right).
\end{align}
In \cite[Lem.\ 32]{Per-typeB}, it was shown that the right-hand side
does not belong to $\{0,-1,-2,\dots\}$.

Now suppose $\alpha=\pm\epsilon_i\in\Phi^{\rs}$ ($i=1,\dots,l$) and
$m\in\ZZ$ such that
$\widetilde{\alpha}=2\alpha+(2m+1)\delta\in\widehat{\Phi}^{\rl}_+$.
We have:
\begin{align}
  ( \lambda_S+\rho,\widetilde{\alpha}^\vee)&=
                                             \frac{2}{4}\left((2m+1)\left(l+\frac{1}{2}\right)+(\overline{\rho}+\mu_S,2\alpha)\right)
                                             =(2m+1)\left(\frac{l}{2}+\frac{1}{4}\right)+(\overline{\rho}+\mu_S,\alpha).
\end{align}
Recalling \eqref{eqn:horizroots}, \eqref{eqn:horizweights}, we see
that $(\mu_S,\alpha)\in \frac{1}{2}\ZZ$.  Recalling
\eqref{eqn:rhol>1}, we see
$(\overline{\rho},\alpha)\in\frac{1}{2}\ZZ$.  Hence,
$( \lambda_S+\rho,\widetilde{\alpha}^\vee)\in \frac{1}{4} +
\frac{1}{2}\ZZ$, and thus not in $\{0,-1,-2,\dots\}$.

The proof for checking the second condition of Definition
\ref{def:admissible} (recall Remark \ref{rem:adm2})) is also similar
to \cite{Per-typeB}.  For $i=1,\dots,k$, denote the coefficient of
$\omega_{i_j}$ in $\mu_S$ by $x_{i_j}\in \frac{1}{2}+\ZZ$.

Using \eqref{eqn:horizroots} and \eqref{eqn:horizweights}, it is easy
to see that for $i\in \{1,\dots,l\}\backslash S$,
$(\lambda_S,\alpha_i^\vee)=(\mu_S,\alpha_i^\vee)=0$.  If $i_j\in S$,
$\delta-\alpha_{i_j}\in \widehat{\Phi}^\ri$.  We have, again using
\eqref{eqn:horizroots} and \eqref{eqn:horizweights}:
\begin{align}
  (\lambda_S,(\delta-\alpha_{i_j})^\vee)=\frac{2}{2}(\lambda_S,\delta-\alpha_{i_j})
  =\left( \left(-l-\frac{1}{2}\right) - (\mu_S,\alpha_{i_j}) \right)=-l-\frac{1}{2}-x_{i_j}\in\ZZ.
\end{align}
Now, if $S=\{i_1,\dots,i_k\}$ has two or more elements, consider
$i_j\in S$ with $j=1,\dots,k-1$.  Note,
$\epsilon_{i_j}-\epsilon_{(i_{j+1}+1)}=\alpha_{i_j}+\alpha_{i_{j}+1}+\alpha_{i_{j}+2}+\cdots+\alpha_{i_{j+1}}\in\widehat{\Phi}^\ri$.
\begin{align}
  (\lambda_S,(\alpha_{i_j}+\alpha_{i_{j}+1}+\alpha_{i_{j}+2}+\cdots+\alpha_{i_{j+1}})^\vee)
  =(\mu_S,\epsilon_{i_j}-\epsilon_{i_{j+1}+1})=x_{i_j}+x_{i_{j+1}}\in\ZZ.
\end{align}
If $S$ has two or more elements, the observations above are enough to
guarantee the second condition of admissibility.  If $S$ has exactly
one element, $S=\{i_1\}$, consider
$\epsilon_i=\alpha_{i_1}+\alpha_{i_1+1}\cdots +\alpha_{l}\in
\widehat{\Phi}^\rs$.  We have:
\begin{align}
  (\lambda_{\{i_1\}},(\alpha_{i_1}+\alpha_{i_1+1}\cdots +\alpha_{l})^\vee)
  =2(\mu_{\{i_1\}},\epsilon_{i_1})=2x_{i_1}\in\ZZ.
\end{align}
This, combined with the other observations is enough to handle the
present case.  Finally, if $S$ is empty, consider
$\delta-\alpha_l=\widehat{\Phi}^\rs$:
\begin{align}
  (\lambda_\phi,(\delta-\alpha_l)^\vee)
  =2(\lambda_\phi,\delta-\alpha_{l})
  =2\left(\left(-l-\frac{1}{2}\right) - (\mu_\phi,\alpha_l) \right)
  =-2l-1\in\ZZ.
\end{align}
\end{proof}

\subsection{Semi-simplicity}
Again, our proofs are parallel to the ones in \cite{AdaMil-sl2},
\cite{Per-typeA}, \cite{Per-typeB} etc., with statements modified to
accommodate the twist.  Recall the notion of category $\sO$ for
representations of affine Kac-Moody algebras, \cite[Ch.\ 9]{Kac-book}.
\begin{thm}(\cite[Thm.\ 4.1]{KacWak-modinv})\label{thm:KWirred} Let
  $\la{g}$ be any affine Lie algebra and let $M$ be a $\la{g}$-module
  from category $\sO$ such that its every irreducible subquotient
  $L(\lambda)$ with highest weight $\lambda$ satisfies:
  \begin{enumerate}
  \item $(\lambda + \rho, \alpha^\vee) \not\in\{-1,-2,\dots,\}$ for
    all $\alpha^\vee\in\widehat{\Delta}^{\vee,\re}_+$ and
  \item $\Re(\lambda+\rho,c)>0$.
  \end{enumerate}
  Then $M$ is completely reducible.
\end{thm}

It is clear that our weights $\lambda_S,\lambda_S'$ for all $l\geq 1$
and $S\subseteq\{1,\dots,l-1\}$ satisfy these conditions.

\begin{thm}(cf.\ \cite[Thm.\ 33]{Per-typeB}) Let $M$ be a weak
  $\nu$-twisted $L(\ourla,\ourlev)$-module that is in category $\sO$
  as a $A_{2l}^{(2)}$-module.  Then, $M$ is completely reducible.
  \label{thm:catOss}
\end{thm}
\begin{proof}
  Any irreducible subquotient $L$ of $M$ is also a $\nu$-twisted
  $L(\ourla,\ourlev)$-module that is in category $\sO$ as a
  $A_{2l}^{(2)}$-module.
  Thus, the highest weight of $L$ is $\lambda_S$ or $\lambda_S'$, in
  particular it satisfies the conditions of Theorem \ref{thm:KWirred}.
  So, $M$ is completely reducible as a $A_{2l}^{(2)}$-module, and thus
  completely reducible as a (weak) $\nu$-twisted
  $L(\ourla,\ourlev)$-module.
\end{proof}

\begin{thm}(cf.\ \cite[Lem.\ 26]{Per-typeB}) Let $M$ be an ordinary
  $\nu$-twisted $L(\ourla,\ourlev)$-module. Then, $M$ is in category
  $\sO$ as a $A_{2l}^{(2)}$-module, in particular, $M$ is completely
  reducible.
	\label{thm:catOrdss}
\end{thm}
\begin{proof}
  $M$ is a level $\ourlev$ module for $A_{2l}^{(2)}$
  \cite{Li-twisted}, in particular, the central element $c$ of
  $A_{2l}^{(2)}$ acts semi-simply on $M$.  Clearly, every conformal
  weight space of $M$ which is finite dimensional by assumption is a
  module for $\la{h}^0$. Thus, $\la{h}^0$ acts semi-simply on $M$ with
  finite dimensional weight spaces. If $v$ is a highest weight vector
  in $M$ of weight $\lambda\in\la{H}^*$, then the irreducible
  $A_{2l}^{(2)}$ module $L(\lambda)$ is an irreducible subquotient of
  $M$, and hence an ordinary $\nu$-twisted
  $L(\ourla,\ourlev)$-module. $L(\lambda)$ has a finite dimensional
  lowest conformal weight space, in particular, this space is finite
  dimensional irreducible module for $\la{g}^0$.  Thus, $\lambda$ has
  only two choices, $\lambda_\phi$ or $\lambda_{\phi}'$ since
  $\mu_\phi$ and $\mu_\phi'$ are the only dominant integral weights
  for $\la{g}^0$ among the possible highest weights (Remark
  \ref{rem:ordsimples}).  This implies that any weight of $M$ has to
  be dominated by one of $\lambda_\phi$ or $\lambda_{\phi'}$, i.e.,
  $\mathrm{wt}(M)\subseteq D(\lambda_\phi)\cup D(\lambda_\phi')$.
  This proves that $M$ is in category $\sO$ as a
  $A_{2l}^{(2)}$-module. The last assertion is due to Theorem
  \ref{thm:catOss}.
\end{proof}

\bibliographystyle{abbrv}

\providecommand{\oldpreprint}[2]{\textsf{arXiv:\mbox{#2}/#1}}\providecommand{\preprint}[2]{\textsf{arXiv:#1
		[\mbox{#2}]}}

\end{document}